\documentclass[a4paper, 11pt]{amsart}   
\usepackage{mathptmx, amssymb,amscd,latexsym, eulervm}   
\usepackage{amsmath}
\usepackage{amsthm}
\usepackage{mathdots}
\usepackage[pagebackref,colorlinks=true,linkcolor=blue,urlcolor=blue]{hyperref}
\usepackage{color}
\usepackage[onehalfspacing]{setspace}
\usepackage{tabularx}
\usepackage{amsfonts}
\usepackage{paralist}
\usepackage{aliascnt}
\usepackage[initials, lite]{amsrefs}
\usepackage{amscd}
\usepackage{blkarray}
\usepackage{mathbbol}
\usepackage{setspace}
\usepackage[inner=2.5cm,outer=2.5cm, bottom=3.2cm]{geometry}
\usepackage{tikz, tikz-cd}
\usepackage{calligra,mathrsfs}

\usepackage{tikz}
\usetikzlibrary{matrix}
\usetikzlibrary{arrows,calc}
\allowdisplaybreaks

\BibSpec{collection.article}{%
	+{}  {\PrintAuthors}                {author}
	+{,} { \textit}                     {title}
	+{.} { }                            {part}
	+{:} { \textit}                     {subtitle}
	+{,} { \PrintContributions}         {contribution}
	+{,} { \PrintConference}            {conference}
	+{}  {\PrintBook}                   {book}
	+{,} { }                            {booktitle}
	+{,} { }                            {series}
	+{, vol.} { }                            {volume}
	+{,} { }                            {publisher}
	+{,} { \PrintDateB}                 {date}
	+{,} { pp.~}                        {pages}
	+{,} { }                            {status}
	+{,} { \PrintDOI}                   {doi}
	+{,} { available at \eprint}        {eprint}
	+{}  { \parenthesize}               {language}
	+{}  { \PrintTranslation}           {translation}
	+{;} { \PrintReprint}               {reprint}
	+{.} { }                            {note}
	+{.} {}                             {transition}
	+{}  {\SentenceSpace \PrintReviews} {review}
}
\AtBeginDocument{%
	\def\MR#1{}
}

\makeatletter
\@namedef{subjclassname@2020}{%
	\textup{2020} Mathematics Subject Classification}
\makeatother

\newcommand{\kk}{\mathbb{k}}

\newcommand{\fP}{\mathfrak{P}}

\newcommand{\ZZ}{{\normalfont\mathbb{Z}}}

\newcommand{\sT}{\mathscr{T}}
\newcommand{\sX}{\mathscr{X}}
\newcommand{\sM}{\mathscr{M}}
\newcommand{\sF}{\mathscr{F}}

\newcommand{\PP}{{\normalfont\mathbb{P}}}

\newcommand{\degIm}{{\normalfont\text{degIm}}}
\newcommand{\mm}{{\normalfont\mathfrak{m}}}

\newcommand{\SSS}{\mathbb{S}}
\newcommand{\length}{{\rm length}}
\newcommand{\QQ}{\mathbb{Q}}
\newcommand{\pp}{\mathfrak{p}}
\newcommand{\aaa}{\mathfrak{a}}
\newcommand{\bb}{\mathfrak{b}}
\newcommand{\qqq}{\mathfrak{q}}
\newcommand{\ttt}{{\normalfont\textbf{t}}}
\newcommand{\nn}{{\mathfrak{N}}}
\newcommand{\nnn}{\mathfrak{n}}

\newcommand{\bn}{{\normalfont\mathbf{n}}}
\newcommand{\bm}{{\normalfont\mathbf{m}}}
\newcommand{\ee}{{\normalfont\mathbf{e}}}
\newcommand{\Ker}{\normalfont\text{Ker}}

\newcommand{\Quot}{{\normalfont\text{Quot}}}

\newcommand{\IM}{\normalfont\text{Im}}
\newcommand{\HT}{\normalfont\text{ht}}

\newcommand{\Ann}{\normalfont\text{Ann}}
\newcommand{\Supp}{\normalfont\text{Supp}}
\newcommand{\Ass}{\normalfont\text{Ass}}

\newcommand{\Rees}{\mathscr{R}}

\newcommand{\OO}{\mathcal{O}}

\newcommand{\LL}{\mathbb{L}}

\newcommand{\FF}{\mathcal{F}}

\newcommand{\HH}{\normalfont\text{H}}

\newcommand{\gr}{\normalfont\text{gr}}

\newcommand{\iniTerm}{\normalfont\text{in}}

\newcommand{\Proj}{\normalfont\text{Proj}}
\newcommand{\Spec}{{\normalfont\text{Spec}}}

\newcommand{\sG}{\mathscr{G}}

\newcommand{\multProj}{\normalfont\text{MultiProj}}
\newcommand{\biProj}{{\normalfont\text{BiProj}}}

\newtheorem{theorem}{Theorem}[section]

\newtheorem{headthm}{Theorem}

\newaliascnt{headcor}{headthm}

\aliascntresetthe{headcor}

\newaliascnt{headconj}{headthm}

\aliascntresetthe{headconj}

\newaliascnt{corollary}{theorem}
\newtheorem{corollary}[corollary]{Corollary}
\aliascntresetthe{corollary}

\newaliascnt{claim}{theorem}

\aliascntresetthe{claim}

\newaliascnt{lemma}{theorem}
\newtheorem{lemma}[lemma]{Lemma}
\aliascntresetthe{lemma}

\newaliascnt{conjecture}{theorem}

\aliascntresetthe{conjecture}

\newaliascnt{proposition}{theorem}
\newtheorem{proposition}[proposition]{Proposition}
\aliascntresetthe{proposition}

\theoremstyle{definition}
\newaliascnt{definition}{theorem}
\newtheorem{definition}[definition]{Definition}
\aliascntresetthe{definition}

\newaliascnt{notation}{theorem}
\newtheorem{notation}[notation]{Notation}
\aliascntresetthe{notation}

\newaliascnt{example}{theorem}
\newtheorem{example}[example]{Example}
\aliascntresetthe{example}

\newaliascnt{examples}{theorem}

\aliascntresetthe{examples}

\newaliascnt{remark}{theorem}
\newtheorem{remark}[remark]{Remark}
\aliascntresetthe{remark}

\newaliascnt{question}{theorem}

\aliascntresetthe{question}

\newaliascnt{questions}{theorem}

\aliascntresetthe{questions}

\newaliascnt{problem}{theorem}

\aliascntresetthe{problem}

\newaliascnt{construction}{theorem}

\aliascntresetthe{construction}

\newaliascnt{setup}{theorem}
\newtheorem{setup}[setup]{Setup}
\aliascntresetthe{setup}

\newaliascnt{algorithm}{theorem}

\aliascntresetthe{algorithm}

\newaliascnt{observation}{theorem}

\aliascntresetthe{observation}

\newaliascnt{defprop}{theorem}
\newtheorem{defprop}[defprop]{Definition-Proposition}
\aliascntresetthe{defprop}

\DeclareFontFamily{OT1}{pzc}{}
\DeclareFontShape{OT1}{pzc}{m}{it}{<-> s * [1.100] pzcmi7t}{}
\DeclareMathAlphabet{\mathchanc}{OT1}{pzc}{m}{it}

\def\equationautorefname~#1\null{(#1)\null}
\def\sectionautorefname~#1\null{Section #1\null}
\def\subsectionautorefname~#1\null{\S #1\null}

\def\trdeg{{\rm trdeg}}
\def\surjects{\twoheadrightarrow}

\begin{document}
	
	\title{Multidegrees, families, and integral dependence}

	\author{Yairon Cid-Ruiz}
	\address{Department of Mathematics, North Carolina State University, Raleigh, NC 27695, USA}
	\email{ycidrui@ncsu.edu}
	
	\author{Claudia Polini}
	\address{Department of Mathematics, University of Notre Dame, Notre Dame, IN 46556, USA}
	\email{cpolini@nd.edu}
	
	\author{Bernd Ulrich}
	\address{Department of Mathematics, Purdue University, West Lafayette, IN 47907, USA}
	\email{bulrich@purdue.edu}
	
	\date{\today}
	\keywords{multidegrees, families, integral dependence, mixed multiplicities, polar multiplicities, Segre numbers, rational maps, projective degrees, specialization}
	\subjclass[2020]{13H15, 14C17, 13B22, 13D40, 13A30}
	
	\maketitle

	\begin{abstract}
		We study the behavior of multidegrees in families and the existence of numerical criteria to detect integral dependence. 
		We show that mixed multiplicities of modules are upper semicontinuous functions when taking fibers and that projective degrees of rational maps are lower semicontinuous under specialization.
		We investigate various aspects of the polar multiplicities and Segre numbers of an ideal and introduce a new invariant that we call polar-Segre multiplicities. 
		In terms of polar multiplicities and our new invariants, we provide a new integral dependence criterion for certain families of ideals. 
		By giving specific examples, we show that the Segre numbers are the only invariants among the ones we consider that can detect integral dependence.
		Finally,  we generalize the result of Gaffney and Gassler regarding the lexicographic upper semicontinuity of Segre numbers.
	\end{abstract}

	\section{Introduction}
	
	This paper is concerned with \emph{the behavior of multidegrees in families} and with \emph{the search for criteria to detect integral dependence}. 
	Although these two themes are not typically studied together, the backbone of our work is the delicate interplay between them.
	Multidegrees provide the natural generalization of the degree of a projective variety to a multiprojective setting, and their study goes back to classical work of van der Waerden \cite{VAN_DER_WAERDEN}.
	The notion of multidegrees (or mixed multiplicities) has become of importance in several areas of mathematics (e.g., algebraic geometry, commutative algebra and combinatorics; see \cites{Bhattacharya,VERMA_BIGRAD,HERMANN_MULTIGRAD,Trung2001,MS, EXPONENTIAL_VARIETIES,KNUTSON_MILLER_SCHUBERT, POSITIVITY, CDNG_MINORS, CDNG_CS_IDEALS,  Huh12, LORENTZIAN, CS_PAPER, TRUNG_VERMA,  KNUTSON_MILLER_YONG,STURMFELS_UHLER}).
	On the other hand, the idea of detecting integral dependence with numerical invariants was initiated with seminal work of Rees \cite{REES}.
	Considerable effort has been made to extend Rees' theorem to the case of arbitrary ideals, modules, and, more generally, algebras (see \cites{KLEIMAN_THORUP_GEOM,KLEIMAN_THORUP_MIXED, UV_NUM_CRIT, SUV_MULT, UV_CRIT_MOD, BOEGER, FLENNER_MANARESI, GAFFNEY, GG, PTUV, cidruiz2024polar}). 
	
	Teissier's Principle of Specialization of Integral Dependence (PSID) can be seen as the first indication of a fruitful connection between the behavior of multiplicities in families and the detection of integral dependence (see \cite{TEISSIER_CYC}, \cite[Appendice I]{TEISSIER_RES2}).
	Indeed, in an analytic setting, the original PSID states that for a family of zero-dimensional ideals with constant Hilbert-Samuel multiplicity, a section is integrally dependent on the total family if and only if it is integrally dependent on the fibers corresponding to a Zariski-open dense subset of the base. 
	For families of not necessarily zero-dimensional ideals (also in an analytic setting), the PSID was extended by Gaffney and Gassler \cite{GG} using Segre numbers.
	This paper continues the line of research traced by the aforementioned works of Teissier, Gaffney and Gassler.
	
	\smallskip
	
	We now describe the results of this paper more precisely. 
	
	\smallskip
	
	\noindent
	1.1.
	\textbf{The behavior of multidegrees and projective degrees of rational maps in families.}
	\smallskip
	
	Let $\kk$ be a field and $T$ be a finitely generated standard $\ZZ_{\ge 0}^p$-graded $\kk$-algebra.
	Let $X = \multProj(T)$ be the corresponding multiprojective scheme embedded in a product of projective spaces $\PP = \PP_{\kk}^{m_1} \times_\kk \cdots \times_\kk \PP_{\kk}^{m_p}$.
	For each $\bn = (n_1,\ldots,n_p) \in \ZZ_{\ge0}^p$ with $|\bn|=n_1+\cdots+n_p=\dim(X)$, one denotes by $\deg_\PP^\bn(X)$ the \emph{multidegree of $X$ of type $\bn$ with respect to $\PP$}.
	In classical geometric terms, if $\kk$ is algebraically closed, then $\deg_\PP^\bn(X)$ is equal to the number of points (counted with multiplicities) in the intersection of $X$ with the product $L_1 \times_\kk \cdots \times_\kk L_p  \subset \PP$, where $L_i \subseteq \PP_\kk^{m_i}$ is a general linear subspace of dimension $m_i-n_i$ for each $1 \le i \le p$.
	More generally, given a finitely generated $\ZZ^p$-graded $T$-module $M$, one denotes by $e(\bn;M)$ the \emph{mixed multiplicity of $M$ of type $\bn$}, for each $\bn \in \ZZ_{\ge 0}^p$ with $|\bn|=\dim(\Supp_{++}(M))$.
	Let $\Psi : \PP_\kk^r \dashrightarrow \PP_\kk^s$ be a rational map and $\Gamma \subset \PP_\kk^r \times_\kk \PP_\kk^s$ be the closure of the graph of $\FF$.
	For each $0 \le i \le r$, the $i$-th \textit{projective degree} of $\Psi:\PP_\kk^r \dashrightarrow \PP_\kk^s$ is given by the following multidegree
	$$
	d_i(\Psi) \,=\, \deg_{\PP_\kk^r \times_\kk \PP_\kk^s}^{i,r-i}\left(\Gamma\right).
	$$
	For more details on these notions, see \autoref{sect_prelim}.
	
	\smallskip
	
	In \autoref{sect_mixed_mult}, we study the behavior of mixed multiplicities under the process of taking fibers with respect to a base ring.
	The idea of studying the multiplicities of families is now classical (see, e.g., \cite{LIPMAN_EQUIMULT}, \cite[Chapter 5]{COX_KPU}), but the case of mixed multiplicities does not seem to have been considered before.
	Let $A$ be a Noetherian ring and $\sT$ be a finitely generated standard $\ZZ_{\ge 0}^p$-graded $A$-algebra.
	Denote by $\sX = \multProj(\sT)$ the corresponding multiprojective scheme.		
	Let $\sM$ be a finitely generated $\ZZ^p$-graded $\sT$-module.
	For each $\pp \in \Spec(A)$, let $\kappa(\pp) = A_\pp/\pp A_\pp$ be the residue field of $\pp$, and consider the finitely generated $\ZZ^p$-graded module $\sM \otimes_A \kappa(\pp)$ over the finitely generated standard $\ZZ_{\ge 0}^p$-graded $\kappa(\pp)$-algebra $\sT(\pp) = \sT \otimes_A \kappa(\pp)$.
	Then, for all $\bn \in \ZZ_{\ge 0}^p$, we introduce a function 
	$$
	e_\bn^\sM  \,:\, \Spec(A) \rightarrow \ZZ \cup \{\infty\}
	$$
	that naturally measures the mixed multiplicities of the fibers $\sM \otimes_A \kappa(\pp)$ (see \autoref{def_mixed_mult_fib}).
	Our first main result is the following.

	\begin{headthm}[\autoref{thm_mixed_mult_fib}]
		For all $\bn \in \ZZ_{\ge 0}^p$, the function $e_\bn^\sM : \Spec(A) \rightarrow \ZZ \cup \{\infty\}$ is upper semicontinuous.
	\end{headthm}
	
	A direct consequence of the above theorem is the upper semicontinuity of the respective functions
	$$
	\deg_{\sX,\PP_A}^\bn : \Spec(A) \rightarrow \ZZ \cup \{\infty\}
	$$
	measuring the multidegrees of the fibers $\sX \times_{\Spec(A)} \Spec(\kappa(\pp))$ (see \autoref{def_multdeg_fib} and \autoref{cor_multdeg_fib}).
	On the other hand, under certain conditions, in \autoref{cor_specializ_mods} we show that mixed multiplicities are lower semicontinuous under the process of taking specializations.

	\smallskip
	
	We study rational maps and their specializations in \autoref{sect_rat_maps}.
	In particular, we are interested in the behavior of projective degrees with respect to specializations. 
	Since projective degrees are the mixed multiplicities of the corresponding Rees algebra (i.e., the multidegrees of the graph), this type of question can be traced back to the problem of specializing Rees algebras (see \cite{EH_SPECIALIZ, KSU}).
	More recently, specializations of rational maps were studied in \cite{SPECIALIZATION_ARON, SPECIALIZATION_MARC_ARON}.  
	Let $A$ be a Noetherian domain, $S = A[x_0,\ldots,x_r]$ be a standard graded polynomial ring and $\PP_A^r = \Proj(S)$.
	Let $\FF :  \PP_A^r \dashrightarrow \PP_A^s$ be a rational map with representative $\mathbf{f} = (f_0:\cdots:f_s)$ such that $\{f_0,\ldots,f_s\} \subset S$ are homogeneous elements of the same degree. 
	Denote by $I = (f_0,\ldots,f_s) \subset S$ the base ideal of $\FF$.
	For any $\pp \in \Spec(A)$, we get the rational map 
	$\FF(\pp) : \PP_{\kappa(\pp)}^r \dashrightarrow \PP_{\kappa(\pp)}^s$ with representative 
	$
	\pi_\pp(\mathbf{f}) = \left( \pi_\pp(f_0) : \cdots : \pi_\pp(f_s) \right)
	$
	where $\pi_\pp(f_i)$ is the image of $f_i$ under the natural map $\pi_\pp : S \rightarrow S(\pp) = S \otimes_A \kappa(\pp)$.
	Then, we introduce the functions 
	\begin{align*}
		{\rm degIm}^\FF : \Spec(A) \rightarrow \ZZ & \qquad\qquad \text{(measures the degree of the image of $\FF(\pp)$)}\\
		d_i^\FF : \Spec(A) \rightarrow \ZZ & \qquad\qquad \text{(measures the projective degrees of $\FF(\pp)$)}\\	
		j^I : \Spec(A) \rightarrow \ZZ & \qquad\qquad \text{(measures the $j$-multiplicity of $I(\pp) \subset S(\pp)$)};
	\end{align*}
	see \autoref{def_spec_rat_map}.
	Our second main result deals with the behavior of the last three functions.

	\begin{headthm}[\autoref{thm_specialization_rat_map}]
		\label{thmB}
		The following statements hold:
		\begin{enumerate}[\rm (i)]
			\item $\degIm^\FF : \Spec(A) \rightarrow \ZZ$ is a lower semicontinuous function.
			\item $d_i^\FF : \Spec(A) \rightarrow \ZZ$ is a lower semicontinuous function for all $0 \le i \le r$.
			\item $j^I : \Spec(A) \rightarrow \ZZ$ is a lower semicontinuous function.
		\end{enumerate}
	\end{headthm}
	
	In \autoref{cor_upper_bounds_proj_degs}, we use \autoref{thmB} to give sharp upper bounds for the projective degrees of several families of rational maps (the list includes perfect ideals of height two and Gorenstein ideals of height three).
	The basic idea is that for several families of rational maps under generic conditions we can compute projective degrees, and then \autoref{thmB} yields an upper bound for any specialization.
	
	\medskip
	
	\noindent
	1.2.
	\textbf{Polar multiplicities, Segre numbers and integral dependence.}
	\smallskip
	
	Our next objective is to study various aspects of the \emph{polar multiplicities} and \emph{Segre numbers} of an ideal and introduce a new invariant that plays an important role in our work.
	One technical goal of our work is to extend several of the results of Gaffney and Gassler \cite{GG} from their analytic setting to an algebraic one over a Noetherian local ring.
	Let $(R,\mm,\kappa)$ be a Noetherian local ring with maximal ideal $\mm$ and residue field $\kappa$.
	Let $d = \dim(R)$, $X = \Spec(R)$ and $I \subset R$ be a proper ideal. 
	Consider the blow-up $\pi : P = \Proj(\Rees(I)) \rightarrow X$ and the exceptional divisor $E = \Proj(\gr_I(R))$ of $X$ along $I$.
	Following the general notion of polar multiplicities due to Kleiman and Thorup \cite{KLEIMAN_THORUP_GEOM,KLEIMAN_THORUP_MIXED}, one defines 
	$$
	\text{(\emph{polar multiplicity}) \;}m_i(I, R) \;=\; m_d^{d-i}(\Rees(I))  \quad \text{ and } \qquad \text{(\emph{Segre number}) \;}c_i(I, R) \;=\; m_{d-1}^{d-i}(\gr_I(R))
	$$
	as polar multiplicities of $\Rees(I)$ and $\gr_I(R)$, respectively. 
	It should be mentioned that the polar multiplicities of a standard graded $R$-algebra can be seen as the multidegrees of a biprojective scheme over the residue field $\kappa$.
	We introduce the new invariant 
	$$
	\nu_i(I, R) \;=\; m_i(I, R) + c_i(I, R)
	$$
	that we call \emph{polar-Segre multiplicity}.
	
	\smallskip
	
	\autoref{sect_invariants_cycles} is dedicated to establish several properties of the invariants $m_i(I, R)$, $c_i(I, R)$ and $\nu_i(I,R)$.
	Let $\delta = o(I)$ be the order of $I$ (see \autoref{nota_G_param}).
	In \autoref{prop_decomp_div}, we show the inequality 
	$$
	\delta \cdot m_{i-1}(I, R) \;\le\; m_i(I, R) + c_{i}(I, R) \;=\; \nu_i(I, R),
	$$
	and that equality holds for all $1 \le i \le d$ if and only if $I$ satisfies the \emph{$\sG$-parameter condition generically} (see \autoref{nota_G_param}).
	Next, we express all these numbers as the multiplicities of the push-forward via $\pi$ of certain cycles obtained by making general cuttings; thus following general tradition when studying local invariants (see, e.g.,  \cite{GG}, \cite[\S 8]{KLEIMAN_THORUP_GEOM}). 
	Assume that $\kappa$ is an infinite field, $\underline{H} = H_1,\ldots, H_d$ is a sequence of general hyperplanes, and denote by $\underline{g}=g_1,\ldots,g_d$ the associated sequence of elements in $I$ (see \autoref{nota_gen_hyper}).
	We introduce the following objects: 
	\begin{align*}
		\text{(\emph{polar scheme})} & \qquad  P_i(I, X)  \;=\; P_i^{\underline{H}}(I, X)  \;=\; \pi\left(H_1 \cap \cdots \cap H_i\right) \\
		\text{(\emph{Segre cycle})} & \qquad  \Lambda_i(I, X)  \;=\; \Lambda_i^{\underline{H}}(I, X)  \;=\; \pi_*\left(\big[E \cap H_1 \cap \cdots \cap H_{i-1}\big]_{d-i}\right) \;\in\; Z_{d-i}(X) \\
		\text{(\emph{polar-Segre cycle})} & \qquad  V_i(I, X)  \;=\; V_i^{\underline{H}}(I, X)  \;=\; \pi_*\left(\big[H_1 \cap \cdots \cap H_{i-1} \cap \pi^*D_i\big]_{d-i}\right) \;\in\; Z_{d-i}(X); 	
	\end{align*}
	for more details, see \autoref{setup_segre_cycles}.
	We have the following unifying result.		
	
	\begin{headthm}[\autoref{thm_polar_segre_cycles}]
		{\rm(}$\kappa$ infinite{\rm)}. The following statements hold: 
		\begin{enumerate}[\rm (i)]
			\item $m_i(I, R) = e_{d-i}\left(P_i(I, X)\right)$ and 
			$
			P_i(I, X) \;=\; \Spec\left(R/(g_1,\ldots,g_i):_RI^\infty\right).
			$
			\smallskip
			\item $c_i(I, R) = e_{d-i}\left(\Lambda_ i(I, X)\right)$ and 
			$$
			\Lambda_ i(I, X) \;=\; \sum_{\substack{\pp \in V\left((g_1,\ldots,g_{i-1}):_RI^\infty\right)\\ \pp \in V(I), \; \dim(R/\pp)=d-i}} e\big(I,\, R_\pp/(g_1,\ldots,g_{i-1})R_\pp:_{R_\pp} I^\infty R_\pp\big)\cdot \left[R/\pp\right] \;\in\; Z_{d-i}(X). 
			$$
			\item $\nu_i(I, R) = e_{d-i}(V_i(I, X))$ and 
			$$
			V_i(I, X) \;=\; \left[P_i(I, X)\right]_{d-i} + \Lambda_i(I, X) \;=\; \big[\Spec\left(R/(g_1,\ldots,g_{i-1}):_RI^\infty+g_iR\right)\big]_{d-i} \;\in\, Z_{d-i}(X).
			$$		
		\end{enumerate}		
	\end{headthm}
	
	A particular consequence of the above theorem is that it recovers known formulas for the polar multiplicities and Segre numbers of an ideal (see \autoref{rem_formulas}).
	
	\smallskip
	
	In \autoref{sect_integral_dep}, we introduce new numerical criteria for integral dependence.
	From the main result of \cite{PTUV}, we know that Segre numbers detect integral dependence (also, in an analytic setting, see \cite{GG}).
	In \autoref{thm_psid}, we prove a PSID that is similar to the ones in \cite[Theorem 4.7]{GG} and \cite[Theorem 4.4]{PTUV}.
	Then a driving question for our work is: \emph{can we detect  integral dependence with the invariants $m_i(I, R)$ and $\nu_i(I, R)$?} 
	In the particular case of the polar multiplicities $m_i(I, R)$ this has been a folklore question for many years.
	Here we give a definitive and perhaps unfortunate answer: 
	\begin{itemize}
		\item \emph{``only Segre numbers detect integral dependence''}.
	\end{itemize}
	Indeed, in \autoref{counter_polar} and \autoref{counter_nu}, we provide examples where the invariants $m_i(I, R)$ and $\nu_i(I, R)$ do not detect integral dependence. 
	On the other hand, we have the following criterion where these two invariants can be used.

	\begin{headthm}[\autoref{thm_new_crit_int_dep}]
		Assume that $R$ is equidimensional and universally catenary, and let $I \subset J$ be two $R$-ideals.
		Suppose the following two conditions: 
		\begin{enumerate}[\rm (a)]
			\item $o(I) = o(J)$.
			\item $I$ satisfies the $\sG$-parameter condition generically {\rm(}see \autoref{nota_G_param}{\rm)}.
		\end{enumerate}
		Then the following are equivalent: 
		\begin{enumerate}[\rm (i)]
			\item $J$ is integral over $I$.
			\item $m_i(I, R) = m_i(J,R)$ for all $0 \le i \le d-1$.
			\item $\nu_i(I, R) = \nu_i(J, R)$ for all $0 \le i \le d$.
		\end{enumerate}
	\end{headthm}
	
	An interesting family of ideals where the above theorem applies is that of equigenerated ideals (see \autoref{cor_equigen}).
	
	\smallskip
	
	Finally, we study the behavior of Segre numbers in families.
	More precisely, we generalize the result of Gaffney and Gassler \cite{GG} regarding the lexicographic upper semicontinuity of Segre numbers. 
	We now introduce a suitable algebraic notation to extend their original result expressed in an analytic setting. 
	Let $\iota : A \hookrightarrow R$ be a flat injective homomorphism of finite type of Noetherian rings and assume that $\pi : R \surjects A$ is a section of $\iota$.
	Let $Q = \Ker(\pi)$.
	For each $\pp \in \Spec(A)$, consider the Noetherian local ring $S(\pp) = R(\pp)_{QR(\pp)}$ that we call the \emph{distinguished fiber} of $\pp$ (see \autoref{subsect_lex}).
	Our last main result is the following.

	\begin{headthm}[\autoref{thm_Segre_lex_upper}]
		Assume that for all $\pp \in \Spec(A)$, the fibers $R(\pp)$ are equidimensional of the same dimension $d$ and
		$\HT(I(\pp)) > 0$.
		Then the function 
		$$
		\pp \in \Spec(A) \;\mapsto\; \big(c_1\left(I, S(\pp)\right), c_2\left(I, S(\pp)\right), \ldots, c_d\left(I, S(\pp)\right)\big) \in \ZZ_{\ge 0}^d 
		$$
		is upper semicontinuous with respect to the lexicographic order.
	\end{headthm}
	
	We give related lexicographic upper semicontinuity results for Segre numbers in \autoref{cor_lex_order} and \autoref{cor_lex_upper_two_ideals}.

	\section{Preliminaries and notation}
	\label{sect_prelim}
	
	Here we recall the concepts of mixed multiplicities, multidegrees and projective degrees.
	We also set up the notation that is used throughout the paper.
	Let $p \ge 1$ be a positive integer and, for each $1 \le i \le p$, let $\ee_i \in \ZZ_{\ge 0}^p$ be the $i$-th elementary vector $\ee_i=\left(0,\ldots,1,\ldots,0\right)$.
	If $\bn = (n_1,\ldots,n_p),\, \bm = (m_1,\ldots,m_p) \in \ZZ^p$ are two vectors, we write $\bn \ge \bm$ whenever $n_i \ge m_i$ for all $1 \le i \le p$, and $\bn > \bm$ whenever $n_j > m_j$ for all $1 \le j \le p$.
	We write $\mathbf{0} \in \ZZ_{\ge 0}^p$ for the zero vector $\mathbf{0}=(0,\ldots,0)$.

	Let $\kk$ be a field and $T$ be a finitely generated standard $\ZZ_{\ge 0}^p$-graded algebra over $\kk$, that is,  $\left[T\right]_{\mathbf{0}}=\kk$ and $T$ is finitely generated over $\kk$ by elements of degree $\ee_i$ with $1 \le i \le p$.
	The multiprojective scheme  $X = \multProj(T)$ corresponding to $T$ is given by 
	the set of all multihomogeneous prime ideals in $T$ not containing the irrelevant ideal $\nn := \left([T]_{\ee_1}\right) \cap \cdots \cap \left([T]_{\ee_p}\right)$, that is,
	$$
	X = \multProj(T) := \big\{ \mathfrak{P} \in \Spec(T) \mid \mathfrak{P} \text{ is multihomogeneous and } \mathfrak{P} \not\supseteq \nn \big\},
	$$
	and its scheme structure is obtained by using multihomogeneous localizations (see, e.g., \cite[\S 1]{HYRY_MULTIGRAD}). 
	We embed $X$ as a closed subscheme of a multiprojective space $\PP:=\PP_\kk^{m_1} \times_\kk \cdots \times_\kk \PP_\kk^{m_p}$.
	
	Let $M$ be a finitely generated $\ZZ^p$-graded $T$-module.
	A homogeneous element is said to be \textit{filter-regular on $M$} (with respect to $\nn$; see \cite[Appendix]{STUCKRAD_VOGEL_BUCHSBAUM_RINGS}) if $z \not\in \fP$ for all associated primes $\fP \in \Ass_T(M)$ of $M$ such that $\fP \not\supseteq \nn$.
	In terms of the multiprojective scheme $X$, this means that $zT_\fP$ is a nonzerodivisor on $M_\fP$ for all $\fP \in X$.
	A sequence of homogeneous elements $z_1,\ldots,z_m$ in $T$ is said to be \textit{filter-regular on $M$} (with respect to $\nn$) if $z_j$ is a filter-regular element on $M/\left(z_1,\ldots,z_{j-1}\right)M$ for all $1 \le j \le m$.
	The relevant support of $M$ is given by $\Supp_{++}(M) := \Supp(M) \cap X$.
	There is a polynomial $P_M(\ttt)=P_M(t_1,\ldots,t_p) \in \QQ[\ttt]=\QQ[t_1,\ldots,t_p]$, called the \emph{Hilbert polynomial} of $M$ (see, e.g., \cite[Theorem 4.1]{HERMANN_MULTIGRAD}, \cite[\S 4]{KLEIMAN_THORUP_GEOM}), such that the degree of $P_M$ is equal to $r = \dim\left(\Supp_{++}(M)\right)$ and 
	$$
	P_M(\nu) = \dim_\kk\left([M]_\nu\right) 
	$$
	for all $\nu \in \ZZ^p$ such that $\nu \gg \mathbf{0}$.
	Furthermore, if we write 
	\begin{equation*}
		P_{M}(\ttt) = \sum_{n_1,\ldots,n_p \ge 0} e(n_1,\ldots,n_p)\binom{t_1+n_1}{n_1}\cdots \binom{t_p+n_p}{n_p},
	\end{equation*}
	then $e(n_1,\ldots,n_p) \in \ZZ_{\ge 0}$ for all $n_1+\cdots+n_p = r$. 
	From this, we obtain the following invariants:

	\begin{definition} \label{def_multdeg}
		\begin{enumerate}[(i)]
			\item For $\bn = (n_1,\ldots,n_p) \in \ZZ_{\ge 0}^p$ with $\lvert\bn\rvert=\dim\left(\Supp_{++}(M)\right)$, $e(\bn;M) := e(n_1,\ldots,n_p)$ is the \textit{mixed multiplicity of $M$ of type $\mathbf{n}$}.
			
			\item For $\bn  \in \ZZ_{\ge 0}^p$ with $\lvert\bn\rvert=\dim(X)$, $\deg_\PP^\bn(X):=e(\bn; T)$ is the \textit{multidegree of $X=\multProj(T) \subset \PP$ of type $\bn$ with respect to $\PP$}. 
		\end{enumerate} 
	\end{definition}

	We recall the following basic concepts related to  rational maps. 
	
	\begin{definition}
		Let $\Psi : \PP_\kk^r \dashrightarrow \PP_\kk^s$ be a rational map, $Y \subset \PP_\kk^s$ be the closure of the image of $\Psi$, and $\Gamma \subset \PP_\kk^r \times_\kk \PP_\kk^s$ be the closure of the graph of $\Psi$.
		The rational map  $\Psi$ is \textit{generically finite} if one of the following equivalent conditions is satisfied:
		\begin{enumerate}[(i)]
			\item The field extension $K(Y) \hookrightarrow K(\PP_{\kk}^r)$ is finite, where $K(\PP_{\kk}^r)$ and $K(Y)$ denote the fields of rational functions of $\PP_{\kk}^r$ and $Y$, respectively.
			\item $\dim(Y) = \dim(\PP_{\kk}^r)=r$.
		\end{enumerate} 
		The \textit{degree} of $\Psi$ is defined as
		$
		\deg(\Psi):=\left[K(\PP_{\kk}^r):K(Y)\right]
		$ when $\Psi$ is generically finite.
		Otherwise, by convention, we set  $\deg(\Psi):=0$.
		For each $0 \le i \le r$, the $i$-th \textit{projective degree} of $\Psi:\PP_\kk^r \dashrightarrow \PP_\kk^s$ is given by 
		$$
		d_i(\Psi) \;:=\; \deg_{\PP_\kk^r \times_\kk \PP_\kk^s}^{i,r-i}\left(\Gamma\right).
		$$
	\end{definition}
	
	For more details on projective degrees, the reader is referred to \cite[Example 19.4]{HARRIS} and \cite[\S 7.1.3]{DOLGACHEV}.
	Of particular interest is the $0$-th projective degree $d_0(\Psi)$ as it is equal to the product of the degree of the map times the degree of the image
	\begin{equation}
		\label{eq_deg_formula}
		d_0(\Psi)\;=\;\deg(\Psi) \cdot \deg_{\PP_\kk^s}(Y)
	\end{equation}
	(e.g., this follows from \cite[Theorem 5.4]{MIXED_MULT} and \cite[Theorem 2.4]{MULTPROJ}).

	When we work with families of ideals, we shall use the following notation. 
	
	\begin{notation}
		Let $A$ be a ring and $R$ be an $A$-algebra.
		For any prime $\pp \in \Spec(A)$, let $\kappa(\pp) := A_\pp/\pp A_\pp$ be the residue field of $\pp$, and set $R_\pp := R \otimes_A A_\pp$, $R(\pp) := R \otimes_A \kappa(\pp)$ and $I(\pp) := I R(\pp) \subset R(\pp)$ for any ideal $I \subset R$.
	\end{notation}

	The notion general element will be quite useful in our treatment, thus we recall the following definition. 
	
	\begin{definition}
		Let $R$ be a Noetherian local ring infinite residue field $\kappa$.
		Let $I \subset R$ be a proper ideal generated by elements $f_1,\ldots,f_m \in R$.
		\begin{itemize}[--]
			\item We say that a property $\mathscr{P}$ holds for a \emph{general element} $g \in I$ if there exists a dense Zariski-open subset $U \subset \kappa^e$ such that whenever $g = a_1f_1+\cdots+a_mf_m$ and the image of $(a_1,...,a_m)$ belongs to $U$, then the property $\mathscr{P}$ holds for $g$.
			\item We say that $\underline{g} = g_1, \ldots, g_k$ is a \emph{sequence of general elements} in $I$ if the image of $g_i$ in the ideal $I /(g_1,\ldots,g_{i-1}) \subset R/(g_1,\ldots,g_{i-1})$ is a general element for all $1 \le i \le k$.
			We also say that $\underline{g} = g_1,\ldots,g_k$ are \emph{sequentially general elements} in $I$.
		\end{itemize}
		Given a multihomogeneous ideal $\mathcal{J} \subset T$, we say that an element $z \in T$ is general in $\mathcal{J}$ if its image is general in the localization $\mathcal{J} T_\mathfrak{M}$ where $\mathfrak{M} := \left([T]_{\ee_1}\right) + \cdots + \left([T]_{\ee_p}\right)$. 
	\end{definition}

	\section{The behavior of mixed multiplicities}
	\label{sect_mixed_mult}
	
	In this section, we study the behavior of mixed multiplicities under the processes of taking fibers and performing  specializations, both with respect to a base ring.
	This section revisits and generalizes some results from \cite[Chapter 5]{COX_KPU} and \cite{SPECIALIZATION_MARC_ARON, SPECIALIZATION_ARON}. 
	We fix the following setup throughout this section. 
	
	\begin{setup}
		\label{setup_fibers_specialization}
		Let $A$ be a Noetherian ring and $\sT$ be a finitely generated standard $\ZZ_{\ge 0}^p$-graded algebra over $A$.
		Let $\sX = \multProj(\sT)$ be the corresponding multiprojective scheme.		
	\end{setup}

	Given a topological space $Z$ and a totally ordered set $(\mathfrak{S}, <)$, 
	we say that a function $f : Z \rightarrow \mathfrak{S}$ is \emph{upper semicontinuous} if $\{ z \in Z \mid f(z) \ge s \}$ is a closed subset of $Z$ for each $s \in \mathfrak{S}$;
	on the other hand, a function $f : Z \rightarrow \mathfrak{S}$ is said to be \emph{lower semicontinuous} if $\{ z \in Z \mid f(z) \le s \}$ is a closed subset of $Z$ for each $s \in \mathfrak{S}$.

	An important basic tool in this paper is the topological Nagata criterion (see, e.g., \cite[Theorem 24.2]{MATSUMURA}). 
	
	\begin{remark}[topological Nagata criterion for openness]
		\label{rem_top_Nagata}
		A subset $U \subset \Spec(A)$ is open if and only if the following two conditions are satisfied: 
		\begin{enumerate}[\rm (i)]
			\item If $\qqq \in U$, then $U$ contains a nonempty open subset of $V(\qqq) \subset \Spec(A)$.
			\item If $\pp, \qqq \in \Spec(A)$, $\pp \in U$ and $\pp \supseteq \qqq$, then $\qqq \in U$.
		\end{enumerate}
	\end{remark}
	
	Given a finitely generated $\ZZ^p$-graded $\sT$-module $\sM$, we seek to study the behavior of mixed multiplicities $e(\bn; \bullet)$ when considering the family of modules $\sM \otimes_A \kappa(\pp)$ with $\pp\in\Spec(A)$.
	Notice that $\sM \otimes_A \kappa(\pp)$ is a finitely generated $\ZZ^p$-graded $\sT(\pp)$-module and that $\sT(\pp)$ is a finitely generated standard $\ZZ_{\ge0}^p$-graded algebra over the field $\kappa(\pp)$.
	
	\begin{definition}
		\label{def_mixed_mult_fib}
		We consider the functions 
		$$
		d_{++}^\sM : \Spec(A) \rightarrow \ZZ, \quad \pp \mapsto \dim\left(\Supp_{++}\left(\sM \otimes_A \kappa(\pp)\right)\right)
		$$
		and 
		$$
		e_\bn^\sM : \Spec(A) \rightarrow \ZZ \cup \{\infty\},  \qquad \pp \mapsto \begin{cases}
			e\left(\bn; \sM \otimes_A \kappa(\pp)\right) & \text{ if } |\bn| = d_{++}^\sM(\pp) \\
			0 & \text{ if } |\bn| > d_{++}^\sM(\pp) \\
			\infty & \text{ if } |\bn| < d_{++}^\sM(\pp) 
		\end{cases}
		$$
		for every $\bn \in \ZZ_{\ge 0}^p$.
		We use the natural ordering on the set $\ZZ \cup \{\infty\}$.
	\end{definition}

	The following result extends \cite[Theorem 5.13]{COX_KPU} to a multigraded setting.
	
	\begin{theorem}
		\label{thm_mixed_mult_fib}
		Assume \autoref{setup_fibers_specialization}.
		Let $\sM$ be a finitely generated $\ZZ^p$-graded $\sT$-module.
		Then the following statements hold: 
		\begin{enumerate}[\rm (i)]
			\item $d_{++}^\sM : \Spec(A) \rightarrow \ZZ$ is an upper semicontinuous function.
			\item $e_\bn^\sM : \Spec(A) \rightarrow \ZZ \cup \{\infty\}$ is an upper semicontinuous function for every $\bn \in \ZZ_{\ge0}^p$. 
		\end{enumerate}
	\end{theorem}
	\begin{proof}
		We prove both statements by utilizing the topological Nagata criterion (see \autoref{rem_top_Nagata}) and Grothendieck's Generic Freeness Lemma (see, e.g., \cite[Theorem 24.1]{MATSUMURA}, \cite[Theorem 14.4]{EISEN_COMM}).
		
		Fix elements $d \in \ZZ$, $\bn = (n_1,\ldots,n_p) \in \ZZ_{\ge 0}^p$ and $e \in \ZZ \cup \{\infty\}$.
		We need to show that 
		$$
		U_d := \big\{\pp \in \Spec(A) \mid d_{++}^\sM(\pp) \le d\big\} \quad \text{ and } \quad V_{\bn,e} := \big\{\pp \in \Spec(A) \mid e_{\bn}^\sM(\pp) \le e\big\}
		$$	
		are open subsets of $\Spec(A)$.
		
		First, we verify condition (i) of \autoref{rem_top_Nagata} for both subsets $U_d$ and $V_{\bn,e}$.
		Let $\qqq \in \Spec(A)$ and  $\overline{A} := A/\qqq$.
		The Generic Freeness Lemma applied to the module $\overline{\sM} := \sM/\qqq \sM$ gives a nonzero element $0 \neq a \in \overline{A}$ such that each graded component of $\overline{\sM}_a$ is a finitely generated free $\overline{A}_a$-module.
		It follows that $P_{\sM \otimes_A \kappa(\pp)} = P_{\sM \otimes_A \kappa(\qqq)}$ for every $\pp \in D(a) \subset V(\qqq) \subset \Spec(A)$, which verifies the validity of condition (i) of \autoref{rem_top_Nagata} for both $U_d$ and $V_{\bn,e}$.
		
		Next, we show that condition (ii) of \autoref{rem_top_Nagata} also holds.
		Due to Nakayama's lemma, for any two primes $\pp, \qqq \in \Spec(A)$ with $\pp \supseteq \qqq$, we have that 
		$$
		\dim_{\kappa(\pp)}\left(\left[\sM \otimes_A \kappa(\pp)\right]_\nu\right) = \mu_{A_\pp}\left(\left[\sM \otimes_A A_\pp\right]_\nu\right) \ge \mu_{A_\qqq}\left(\left[\sM \otimes_A A_\qqq\right]_\nu\right) = \dim_{\kappa(\qqq)}\left(\left[\sM \otimes_A \kappa(\qqq)\right]_\nu\right)
		$$
		for all $\nu \in \ZZ^p$.
		The dimension of the relevant support equals the degree of the Hilbert polynomial, and the latter can be read-off from the Hilbert function.
		For any $\pp, \qqq \in \Spec(A)$ with $\pp \supseteq \qqq$, it follows that $P_{\sM\otimes_A \kappa(\pp)}(\nu) \ge P_{\sM\otimes_A \kappa(\qqq)}(\nu)$ for all $\nu \gg \mathbf{0}$, and so $d_{++}^\sM(\pp) \ge d_{++}^\sM(\qqq)$. 
		This shows that condition (ii) of \autoref{rem_top_Nagata} is satisfied for the subset $U_d$.
		
		Given two primes $\pp, \qqq \in \Spec(A)$ with $\pp \supseteq \qqq$ and $d_{++}^\sM(\pp) > d_{++}^\sM(\qqq)$, we easily check from the definition of the function $e_\bn^\sM$ that $e_\bn^\sM(\pp) \ge e_\bn^\sM(\qqq)$.
		
		Next, consider the case $\pp, \qqq \in \Spec(A)$ with $\pp \supseteq \qqq$ and $d_{++}^\sM(\pp) = d_{++}^\sM(\qqq)$. 
		(That $P_{\sM\otimes_A \kappa(\pp)}(\nu) - P_{\sM\otimes_A \kappa(\qqq)}(\nu) \ge 0$ for all $\nu \gg \mathbf{0}$ does not necessarily imply that the coefficients of the monomials of highest degree of $P_{\sM\otimes_A \kappa(\pp)}$ are bigger or equal than the ones of $P_{\sM\otimes_A \kappa(\qqq)}$; for instance, $f(x,y) = (x-y)^2= x^2 - 2xy + y^2 \in \QQ[x,y]$.)
		
		Let $r := |\bn|$.
		We only need to consider the case where $r = d_{++}^\sM(\pp) = d_{++}^\sM(\qqq)$.
		Notice that we can reduce modulo $\qqq$ and localize at $\pp$. 
		Hence we assume $A$ is a local domain with maximal ideal $\pp$ and $\qqq = 0$.
		By utilizing the faithfully flat extension $A \rightarrow A[y]_{\pp A[y]}$, we may assume that the residue field of $A$ is not an algebraic extension of a finite field. 
		From \cite[Lemma 2.6]{PTUV}, for any multihomogeneous ideal $\mathfrak{J} \subset \sT$, we can choose an element $z \in \mathfrak{J}$ whose image is general in both $\mathfrak{J}\sT(\pp)$ and $\mathfrak{J}\sT(\qqq)$.
		Therefore, by prime avoidance, there exists a sequence of homogeneous elements $z_1, \ldots,z_r$ in $\sT$
		such that the following three conditions are satisfied: 
		\begin{enumerate}
			\item each $z_j \in \sT$ has degree $\deg(z_j) = \ee_{l_j} \in  \ZZ_{\ge 0}^p$ where $1 \le l_j \le p$;
			\item $n_i$ equals the number $|\{  j  \mid 1 \le j \le r \text{ and } l_j = i \}|$;
			\item $\lbrace z_1,\ldots,z_r\rbrace \sT(\pp)$ and $\lbrace z_1,\ldots,z_r\rbrace \sT(\qqq)$ are filter-regular sequences on the modules $\sM \otimes_A \kappa(\pp)$ and $\sM\otimes_A \kappa(\qqq)$, respectively.
		\end{enumerate}
		To simplify notation, let $\sM_{j,\pp} := \sM/(z_1,\ldots,z_{j})\sM \otimes_A \kappa(\pp)$.
		By successively applying \cite[Lemma 3.9]{MIXED_MULT}, we obtain
		$$
		e\left(\bn; \sM \otimes_A \kappa(\pp)\right) =  e\left(\mathbf{0}; \sM_{r,\pp}\right)
		\quad\text{ and }\quad
		e\left(\bn; \sM \otimes_A \kappa(\qqq)\right) =  e\left(\mathbf{0}; \sM_{r,\qqq}\right).
		$$	
		We choose $\mathbf{0} \ll \nu \in \ZZ_{\ge0}^p$ with the property that $e\left(\mathbf{0}; \sM_{r,\pp}\right) = \dim_{\kappa(\pp)}\left(\left[\sM_{r,\pp}\right]_\nu\right)$ and $e\left(\mathbf{0}; \sM_{r,\qqq}\right) = \dim_{\kappa(\qqq)}\left(\left[\sM_{r,\qqq}\right]_\nu\right)$.
		Finally, Nakayama's lemma yields that
		\begin{align*}
			&e\left(\bn; \sM \otimes_A \kappa(\pp)\right) = e\left(\mathbf{0}; \sM_{r,\pp}\right) = \dim_{\kappa(\pp)}\left(\left[\sM_{r,\pp}\right]_\nu\right) = \mu_{A_\pp}\left(\left[\sM/(z_1,\ldots,z_r)\sM \otimes_A A_\pp\right]_\nu\right) \\
			&\quad\ge \mu_{A_\qqq}\left(\left[\sM/(z_1,\ldots,z_r)\sM \otimes_A A_\qqq\right]_\nu\right) = \dim_{\kappa(\qqq)}\left(\left[\sM_{r,\qqq}\right]_\nu\right) = e\left(\mathbf{0}; \sM_{r,\qqq}\right) =  e\left(\bn; \sM \otimes_A \kappa(\qqq)\right).
		\end{align*}
		Therefore, the subset $V_{\bn,e}$ satisfies condition (ii) of \autoref{rem_top_Nagata}.
		This completes the proof of the theorem.
	\end{proof}

	We restate the above theorem for the case of multidegrees. 
	As before, we embed $\sX$ as a closed subscheme of a multiprojective space $\PP_A := \PP_A^{m_1} \times_A \cdots \times_A \PP_A^{m_p}$.
	We seek to study the multidegrees of the fibers $\sX_\pp := \sX \times_{\Spec(A)} \Spec(\kappa(\pp)) = \multProj(\sT(\pp)) \subset \PP_\pp := \PP_A \times_{\Spec(A)} \Spec(\kappa(\pp))$. 
	
	\begin{definition}
		\label{def_multdeg_fib}
		We define the functions 
		$$
		d^\sX : \Spec(A) \rightarrow \ZZ, \quad \pp \mapsto \dim\left(\sX_\pp\right)
		$$
		and 
		$$
		\deg_{\sX,\PP_A}^\bn : \Spec(A) \rightarrow \ZZ \cup \{\infty\},  \qquad \pp \mapsto \begin{cases}
			\deg_{\PP_\pp}^\bn(\sX_\pp) & \text{ if } |\bn| = \dim\left(\sX_\pp\right) \\
			0 & \text{ if } |\bn| > \dim\left(\sX_\pp\right) \\
			\infty & \text{ if } |\bn| < \dim\left(\sX_\pp\right)
		\end{cases}
		$$
		for every $\bn \in \ZZ_{\ge0}^p$.
	\end{definition}
	
	We have the following direct consequence of \autoref{thm_mixed_mult_fib}.
	
	\begin{corollary}
		\label{cor_multdeg_fib}
		Assume \autoref{setup_fibers_specialization}.
		Then the following statements hold: 
		\begin{enumerate}[\rm (i)]
			\item $d^\sX : \Spec(A) \rightarrow \ZZ$ is an upper semicontinuous function.
			\item $\deg_{\sX,\PP_A}^\bn : \Spec(A) \rightarrow \ZZ \cup \{\infty\}$ is an upper semicontinuous function for every $\bn \in \ZZ_{\ge0}^p$. 
		\end{enumerate}
	\end{corollary}

	We quickly revisit the notion of specialization of a module. 
	Here we follow the same setting and notations as in \cite{SPECIALIZATION_MARC_ARON}.
	Let $\sM$ be a finitely generated torsionless $\ZZ^p$-graded $\sT$-module, with a fixed injection $\iota : \sM \hookrightarrow \sF$ into a free $\ZZ^p$-graded $\sT$-module of finite rank.
	For any $\pp \in \Spec(A)$, the \emph{specialization of $\sM$ with respect to $\pp$} is defined as 
	$$
	\SSS_\pp(\sM) \,:=\, \IM\Big( \iota \otimes_A \kappa(\pp) : \sM \otimes_A \kappa(\pp) \rightarrow \sF \otimes_A \kappa(\pp)\Big).
	$$
	We now consider the following function 
	$$
	\SSS e_\bn^\sM : \Spec(A) \rightarrow \ZZ,  \qquad \pp \mapsto 
	e\left(\bn; \SSS_\pp(\sM) \right) 
	$$
	for every $\bn \in \ZZ_{\ge0}^p$ with $|\bn| \ge \dim\left(\Supp_{++}\left(\SSS_{\pp}(\sM)\right)\right)$.

	The following result deals with the behavior of mixed multiplicities with respect to specializations.
	
	\begin{corollary}
		\label{cor_specializ_mods}
		Assume \autoref{setup_fibers_specialization} and that each graded component of $\sT$ is a free $A$-module.
		Let $\sM$ be a finitely generated torsionless $\ZZ^p$-graded $\sT$-module, with a fixed injection $\iota : \sM \hookrightarrow \sF$ into a free $\ZZ^p$-graded $\sT$-module of finite rank.
		Let $r$ be the common dimension $\dim(\Supp_{++}(\sF\otimes_A\kappa(\pp)))$ for all $\pp \in \Spec(A)$.
		Then the function 
		$$
		\SSS e_\bn^\sM : \Spec(A) \rightarrow \ZZ 
		$$ 
		is lower semicontinuous for every $\bn \in \ZZ_{\ge0}^p$ with $|\bn| \ge r$. 
	\end{corollary}
	\begin{proof}
		For any $\pp \in \Spec(A)$, we have the short exact sequence 
		$$
		0 \rightarrow \SSS_\pp(\sM) \rightarrow \sF \otimes_A \kappa(\pp) \rightarrow \sF/\sM \otimes_A \kappa(\pp) \rightarrow 0.
		$$
		From the additivity of mixed multiplicities, we get $e_\bn^\sF(\pp) = e_\bn^{\sF/\sM}(\pp) + \SSS e_\bn^\sM(\pp)$.
		Since $e_\bn^\sF$ is a constant function by assumption and $e_{\bn}^{\sF/\sM}$ is upper semicontinuous by \autoref{thm_mixed_mult_fib}, the result of the corollary follows. 
	\end{proof}

	\section{Rational maps and their specializations}
	\label{sect_rat_maps}

	Here we concentrate on a specialization process of rational maps.
	The next setup is now in place. 
	
	\begin{setup}
		\label{setup_specialization_rat_maps}
		Let $r < s$ be two positive integers.
		Let $A$ be a Noetherian domain, $S = A[x_0,\ldots,x_r]$ be a standard graded polynomial ring, $\PP_A^r = \Proj(S)$, and $\mm = (x_0,\ldots,x_r) \subset S$ be the graded irrelevant ideal.
		Let $\FF :  \PP_A^r \dashrightarrow \PP_A^s$ be a rational map with representative $\mathbf{f} = (f_0:\cdots:f_s)$ such that $\{f_0,\ldots,f_s\} \subset S$ are homogeneous elements of degree $\delta > 0$.
	\end{setup}

	We specialize this rational map as follows.
	For any $\pp \in \Spec(A)$, we get the rational map $$\FF(\pp) \,:\, \PP_{\kappa(\pp)}^r \dashrightarrow \PP_{\kappa(\pp)}^s$$ with representative 
	$
	\pi_\pp(\mathbf{f}) = \left( \pi_\pp(f_0) : \cdots : \pi_\pp(f_s) \right)
	$
	where $\pi_\pp(f_i)$ is the image of $f_i$ under the natural map $\pi_\pp : S \rightarrow S(\pp)$.
	
	Let $I = (f_0,\ldots,f_s) \subset S$ be the base ideal of the rational map $\FF : \PP_A^r  \dashrightarrow \PP_A^s$.
	The closure of the graph of $\FF$ is given as $\Gamma = \biProj(\Rees(I)) \subset \PP_A^r \times_A \PP_A^s$ where $\Rees(I) := \bigoplus_{n=0}^\infty I^nT^n \subset S[T]$ is the Rees algebra of $I$.
	As customary, $\Rees(I)$ is presented as a quotient of a standard bigraded polynomial ring $\sT := S \otimes_A A[y_0,\ldots,y_s]$ by using the $A$-algebra homomorphism 
	$$
	\sT \twoheadrightarrow \Rees(I), \quad x_i \mapsto x_i, \; y_j \mapsto f_jt.
	$$
	We have the equalities $I(\pp) = \left( \pi_\pp(f_0), \ldots, \pi_\pp(f_s) \right) \subset S(\pp)$ and $(I^k)(\pp) = I(\pp)^k \subset S(\pp)$ for all $\pp \in\Spec(A)$ and $k \ge 0$.
	For any $\pp \in \Spec(A)$, let $\Gamma(\pp) \subset \PP_{\kappa(\pp)}^r \times_{\kappa(\pp)} \PP_{\kappa(\pp)}^s$ and $Y(\pp) \subset \PP_{\kappa(\pp)}^s$ be the closures of the graph and the image of the rational map $\FF(\pp)$.
	We have that $\dim(\Gamma(\pp)) \le r$ and $\dim(Y(\pp)) \le r$ for all $\pp \in \Spec(A)$.
	We also consider the \emph{j-multiplicity} of ideal $I(\pp) \subset S(\pp)$ for all $\pp \in \Spec(A)$ (see \cite{ACHILLES_MANARESI_J_MULT}, \cite[\S 6.1]{FLENNER_O_CARROLL_VOGEL}).

	\begin{definition}
		\label{def_spec_rat_map}
		We have the following three functions:  
		$
		j^I : \Spec(A) \rightarrow \ZZ, \; \pp \mapsto j\left(I(\pp)\right),
		$
		$$
		\degIm^\FF : \Spec(A) \rightarrow \ZZ, \qquad \pp \mapsto \begin{cases}
			\deg_{\PP_{\kappa(\pp)}^s}\left(Y(\pp)\right) & \text{ if } \dim(Y(\pp)) = r \\
			0 & \text{ if } \dim(Y(\pp)) < r
		\end{cases}
		$$
		and 
		$d_i^\FF : \Spec(A) \rightarrow \ZZ,\;d_i\left(\FF(\pp)\right)$ for all $0 \le i \le r$.
	\end{definition}	
	
	Our main result regarding these functions is the following theorem.

	\begin{theorem}
		\label{thm_specialization_rat_map}
		Assume \autoref{setup_specialization_rat_maps}.
		Then the following statements hold: 
		\begin{enumerate}[\rm (i)]
			\item $\degIm^\FF : \Spec(A) \rightarrow \ZZ$ is a lower semicontinuous function.
			\item $d_i^\FF : \Spec(A) \rightarrow \ZZ$ is a lower semicontinuous function for all $0 \le i \le r$.
			\item $j^I : \Spec(A) \rightarrow \ZZ$ is a lower semicontinuous function.
		\end{enumerate}
	\end{theorem}	
	\begin{proof}
		By \cite[Theorem 5.3]{KPU_blowup_fibers}, we have that $j^I(\pp) = \delta \cdot d_0^\FF(\pp)$ for all $\pp \in \Spec(A)$.
		Thus, we only need to prove parts (i) and (ii) of the theorem. 
		
		Fix $0 \le i \le r$, $e \in \ZZ$ and $h \in \ZZ$. 
		It remains to show that 
		$$
		D_e := \big\{\pp \in \Spec(A) \mid \degIm^\FF(\pp) \ge e\big\} \quad \text{ and } \quad E_{i,h} := \big\{\pp \in \Spec(A) \mid d_i^\FF(\pp) \ge h\big\}
		$$
		are open subsets of $\Spec(A)$.
		Again, to prove this we utilize a combination of the topological Nagata criterion and the Generic Freeness Lemma.
		
		First, we verify condition (i) of \autoref{rem_top_Nagata} for both $D_e$ and $E_{i,h}$.
		Let $\qqq \in \Spec(A)$, and set $\overline{A} = A/\qqq$, $\overline{S} = S/\qqq S$ and $\overline{I} = I\overline{S} \subset \overline{S}$.
		We use the version of the Generic Freeness Lemma  given in \cite[Lemma 8.1]{HOCHSTER_ROBERTS_INVARIANTS} applied to the inclusion of algebras $\Rees_{\overline{S}}\left(\overline{I}\right) \hookrightarrow \overline{S}[t]$, and we find a nonzero element $0 \neq a \in \overline{A}$ such that each graded component of $\overline{S}[t]/\Rees_{\overline{S}}\left(\overline{I}\right) \otimes_{\overline{A}} \overline{A}_a$ is a finitely generated free $\overline{A}_a$-module (also, see \cite[Theorem 3.5]{SPECIALIZATION_MARC_ARON}).
		For every $\pp \in D(a) \subset V(\qqq) \subset \Spec(A)$, we obtain that $\dim_{\kappa(\pp)}\left(\left[I(\pp)^n\right]_\nu\right) = \dim_{\kappa(\qqq)}\left(\left[I(\qqq)^n\right]_\nu\right)$ for all $n \ge 0$ and $\nu \in \ZZ$.
		Accordingly, condition (i) of \autoref{rem_top_Nagata} holds for both $D_e$ and $E_{i,h}$.
		
		Next, we show that condition (ii) of \autoref{rem_top_Nagata} also holds for $D_e$ and $E_{i,h}$.
		
		For any $\pp \in \Spec(A)$, we have that $\degIm^\FF(\pp) = e_{r+1}\left(L_\pp\right)$ and $d_i^\FF(\pp) = e\left(i, r-i; \Rees_{S(\pp)}(I(\pp))\right)$, where $L_\pp := \kappa(\pp)\left[\pi_\pp(f_0), \ldots, \pi_\pp(f_s)\right] = \bigoplus_{n=0}^\infty \left[I(\pp)^n\right]_{n\delta}$.
		For each $n \ge 0$, we have a short exact sequence 
		$$
		0 \rightarrow I(\pp)^n \rightarrow S(\pp) \rightarrow S/I^n \otimes_A \kappa(\pp) \rightarrow 0.
		$$
		Since we know that $\dim_{\kappa(\pp)}\left(\left[S/I^n \otimes_A \kappa(\pp)\right]_\nu\right) \ge \dim_{\kappa(\qqq)}\left(\left[S/I^n \otimes_A \kappa(\qqq)\right]_\nu\right)$ for all $n \ge 0$, $\nu \in \ZZ$ and $\pp, \qqq \in \Spec(A)$ with $\pp \supseteq \qqq$, it follows that $D_e$ satisfies condition (ii) of \autoref{rem_top_Nagata}.
		
		Fix two primes $\pp, \qqq \in \Spec(A)$ with $\pp \supseteq \qqq$.
		As in the proof of \autoref{thm_mixed_mult_fib}, we may assume that $\qqq = 0$ and that $A$ is a local domain with maximal ideal $\pp$.
		Moreover, by making a purely transcendental field extension, 
		we may assume that for any multihomogeneous ideal $\mathfrak{J} \subset \sT$, we can choose an element $z \in \mathfrak{J}$ whose image is general in both $\mathfrak{J}\sT(\pp)$ and $\mathfrak{J}\sT(\qqq)$ (see \cite[Lemma 2.6]{PTUV}).

		By applying \cite[Proposition 5.6]{MIXED_MULT}, we find a sequence $\{z_1, \ldots,z_i\} \subset \mm \subset S \subset \sT$ of homogeneous elements of degree one such that
		$$
		d_i^\FF(\pp) =  e\left(0,r-i;\, \Rees_{R_\pp}(J_\pp)\right)
		\quad \text{ and } \quad 
		d_i^\FF(\qqq) =  e\left(0,r-i;\, \Rees_{R_\qqq}(J_\qqq)\right)
		$$ 
		where $R_\pp := S(\pp)/(z_1,\ldots,z_i)S(\pp)$, $R_\qqq := S(\qqq)/(z_1,\ldots,z_i)S(\qqq)$, $J_\pp := IR_\pp \subset R_\pp$ and $J_\qqq := IR_\qqq \subset R_\qqq$.
		Since $i \le r$, we may further assume that $\{z_1,\ldots,z_i\} S(\pp)$ and $\{z_1,\ldots,z_i\} S(\qqq)$ are regular sequences on the polynomial rings $S(\pp)$ and $S(\qqq)$, respectively.
		
		There exists a positive integer $m > 0$ such that
		$$
		e\left(0,r-i;\, \Rees_{R_\pp}(J_\pp)\right) = \lim_{n \to \infty} \frac{\dim_{\kappa(\pp)}\left(\left[\Rees_{R_\pp}(J_\pp)\right]_{(m,n)}\right)}{n^{r-i}/(r-i)!} = \lim_{n \to \infty} \frac{\dim_{\kappa(\pp)}\left(\left[J_\pp^n\right]_{m+n\delta}\right)}{n^{r-i}/(r-i)!}
		$$
		and 
		$$
		e\left(0,r-i;\, \Rees_{R_\qqq}(J_\qqq)\right) = \lim_{n \to \infty} \frac{\dim_{\kappa(\qqq)}\left(\left[\Rees_{R_\qqq}(J_\qqq)\right]_{(m,n)}\right)}{n^{r-i}/(r-i)!} = \lim_{n \to \infty} \frac{\dim_{\kappa(\qqq)}\left(\left[J_\qqq^n\right]_{m+n\delta}\right)}{n^{r-i}/(r-i)!}.
		$$
		So, to show that $d_i^\FF(\pp) \le d_i^\FF(\qqq)$, it suffices to verify that $\dim_{\kappa(\pp)}\left(\left[J_\pp^n\right]_{\nu}\right) \le \dim_{\kappa(\qqq)}\left(\left[J_\qqq^n\right]_{\nu}\right)$ for all $n \ge 0$ and $\nu \in \ZZ$.
		Notice that the Hilbert functions of $R_\pp$ and $R_\qqq$ are equal.
		Then, from Nakayama's lemma we get 
		\begin{align*}
			\dim_{\kappa(\pp)}\left(\left[R_\pp/J_\pp^n\right]_{\nu}\right) &= \mu_{A_\pp}\left( \left[ \frac{S}{\left(z_1,\ldots,z_i,I^n\right)S} \otimes_A A_\pp \right]_\nu  \right)  \\
			&\ge \mu_{A_\qqq}\left( \left[ \frac{S}{\left(z_1,\ldots,z_i,I^n\right)S} \otimes_A A_\qqq \right]_\nu  \right) = \dim_{\kappa(\qqq)}\left(\left[R_\qqq/J_\qqq^n\right]_{\nu}\right)
		\end{align*}
		for all $n \ge 0$ and $\nu \in \ZZ$.
		Finally, this implies that $d_i^\FF(\pp) \le d_i^\FF(\qqq)$, and so it follows that condition (ii) of \autoref{rem_top_Nagata} holds for $E_{i,h}$.
		So, we are done with the proof of the theorem.  
	\end{proof}

	We single out an important corollary of \autoref{thm_specialization_rat_map}.
	When $r = s$, and we consider a rational map of the form $\FF : \PP_A^r \dashrightarrow \PP_A^r$, we have some control over the degree of the specialized rational maps $\FF(\pp) : \PP_{\kappa(\pp)}^r \dashrightarrow \PP_{\kappa(\pp)}^r$.
	
	\begin{corollary}
		\label{cor_spec_deg_map}
		Assume \autoref{setup_specialization_rat_maps}.
		Let $\FF : \PP_A^r \dashrightarrow \PP_A^r$ be a rational map.
		Then the function
		$$
		\deg^\FF : \Spec(A) \rightarrow \ZZ, \qquad \pp \mapsto \deg\left(\FF(\pp)\right)
		$$
		is lower semicontinuous.
	\end{corollary}
	\begin{proof}
		We have that $\deg^\FF(\pp) = d_0^\FF(\pp)$ for all $\pp \in \Spec(A)$.
		Indeed, if $\dim(Y(\pp)) < r$, both $\deg^\FF(\pp)$ and $d_0^\FF(\pp)$ are equal to zero; and if $\dim(Y(\pp)) = r$, $\deg_{\PP_{\kappa(\pp)}^r}(Y(\pp)) =1$ and so $d_0(\FF(\pp)) = \deg(\FF(\pp))$ according to \autoref{eq_deg_formula}.
		Thus, the result follows from \autoref{thm_specialization_rat_map}. 
	\end{proof}

	We now apply the above results to different families of rational maps. 
	We obtain generalizations of \cite[Theorems 6.3, 6.8]{SPECIALIZATION_ARON}, and we eliminate the conditions assumed there.
	The following corollary yields significant upper bounds for the projective degrees of certain families of rational maps. 
	It should be mentioned that these inequalities are sharp for the general members of the considered families (see \cite[Theorems 5.7, 5.8]{MIXED_MULT}).
	
	\begin{corollary}
		\label{cor_upper_bounds_proj_degs}
		Let $\kk$ be a field, $R = \kk[x_0,\ldots,x_r]$ be a standard graded polynomial ring, $\PP_\kk^r = \Proj(R)$, $\Psi: \PP_\kk^r \dashrightarrow \PP_\kk^s$ be a rational map with representative $\mathbf{g} = (g_0:\cdots:g_s)$ and base ideal $J = (g_0,\ldots,g_s) \subset R$, and suppose that $\delta=\deg(g_j) >0$.
		Then the following statements hold: 
		\begin{enumerate}[\rm (i)]
			\item $d_i(\Psi) \le \delta^{r-i}$ for all $0 \le i \le r$.

			\item Suppose that $J$ is a perfect ideal of height two with Hilbert-Burch resolution of the form
			$$
			0 \rightarrow \bigoplus_{i=1}^sR(-\delta-\mu_i) \rightarrow {R(-\delta)}^{s+1} \rightarrow  J \rightarrow 0.			
			$$
			Then, for all $0 \le i \le r$, we have 
			$$
			d_i(\Psi) \le  
			e_{r-i}(\mu_1,\ldots,\mu_s)
			$$
			where $	e_{r-i}(\mu_1,\ldots,\mu_s)$ denotes the elementary symmetric polynomial
			$$
			e_{r-i}(\mu_1,\ldots,\mu_s)= 
			\sum_{1\le j_1  < \cdots < j_{r-i} \le s} \mu_{j_1}\cdots\mu_{j_{r-i}}.
			$$
			In particular, if $r = s$, then $\deg(\Psi) \le \mu_1\cdots \mu_r$.

			\item Suppose that $J$ is a Gorenstein ideal of height three.
			Let $D \ge 1$ be the degree of every nonzero entry of an alternating minimal presentation matrix of $J$.
			Then, for all $0 \le i \le r$, we have 
			$$
			d_i(\Psi) \le \begin{cases}
				D^{r-i} \sum_{k=0}^{\lfloor\frac{s-r+i}{2}\rfloor}\binom{s-1-2k}{r-i-1} & \text{ if } 0 \le i \le r-3 \\
				\delta^{r-i} & \text{ if } r-2 \le i \le r.
			\end{cases}
			$$
			In particular, if $r=s$, then $\deg(\Psi) \le D^r$.
		\end{enumerate}
	\end{corollary}
	\begin{proof}
		(i) For $0 \le j \le s$, consider a set of variables $\mathbf{z}_j = \{ z_{j,1}, \ldots, z_{j,m} \}$ over $\kk$ with $m = \binom{\delta+r}{r}$, and set $\mathbf{z} = \mathbf{z}_0 \cup \cdots \cup \mathbf{z}_s$.
		Let $A = \kk[\mathbf{z}]$, $S = A[x_0,\ldots,x_r]$ and consider the generic polynomials 
		$$
		G_j \,:=\, z_{j,1} x_0^\delta + z_{j,2} x_0^{\delta-1}x_1 + \cdots + z_{j,m} x_r^\delta \in S. 
		$$
		Let $\FF : \PP_A^r \dashrightarrow \PP_A^s$ be a rational map with base ideal $I = (G_0,\ldots,G_s) \subset S$.
		Let $\xi = (0) \subset \Spec(A)$ be the generic point and $\mm_\alpha = \left(\{z_{j,k} - \alpha_{j,k}\}_{j,k}\right)\in \Spec(A)$ be a rational maximal ideal such that $J = I(\mm_\alpha) \subset R$.
		It is known that the projective degrees of the morphism $\FF(\xi) : \PP_{\kk(\mathbf{z})}^r\rightarrow \PP_{\kk(\mathbf{z})}^s$ (it is base point free as $I \otimes_A \kk(\mathbf{z}) \subset S(\xi) = \kk(\mathbf{z})[x_0,\ldots,x_r]$ is a zero-dimensional ideal) are equal to $d_i\left(\FF(\xi)\right) = \delta^{r-i}$ (see, e.g., \cite[Observation 3.2]{KPU_blowup_fibers}).
		By using \autoref{thm_specialization_rat_map}, we obtain $d_i(\Psi) = d_i^\FF(\mm_\alpha) \le d_i^\FF(\xi) = \delta^{r-i}$ for all $0 \le i \le r$.
		
		(ii) 	For $1 \le j \le s+1$ and $1 \le k \le s$,  let 
		$
		\mathbf{z}_{j,k}=\{z_{j,k,1}, z_{j,k,2}, \ldots, z_{j,k,m_k}\}
		$	
		denote a set of variables over $\kk$ of cardinality $m_k=\binom{\mu_k+r}{r}$, and set $\mathbf{z} = \bigcup_{j,k} \mathbf{z}_{j,k}$.
		Let $A = \kk[\mathbf{z}]$, $S = A[x_0,\ldots,x_r]$ and consider the generic $(s+1)\times s$ Hilbert-Burch matrix 
		$$
		\mathcal{M} = \left( \begin{array}{cccc}
			p_{1,1} & p_{1,2} & \cdots & p_{1,s} \\
			p_{2,1} & p_{2,2} & \cdots & p_{2,s}\\
			\vdots & \vdots & & \vdots\\
			p_{s+1,1} & p_{s+1,2} & \cdots & p_{s+1,s}\\
		\end{array}
		\right)	
		$$	
		where each polynomial $p_{j,k} \in S$ is given by 
		$$
		p_{j,k} = z_{j,k,1} x_0^{\mu_k} + z_{j,k,2} x_0^{\mu_k-1}x_1 + \cdots + z_{j,k,m_k}x_r^{\mu_k}.
		$$
		Let $\FF : \PP_A^r \dashrightarrow \PP_A^s$ be a rational map with base ideal $I = I_s(\mathcal{M}) \subset S$.
		Let $\varphi \in R^{(s+1) \times s}$ be the Hilbert-Burch presentation of $J$.
		Let $\xi = (0) \subset \Spec(A)$ be the generic point and $\mm_\alpha = \left(\{z_{j,k,l} - \alpha_{j,k,l}\}_{j,k,l}\right)\in \Spec(A)$ be a rational maximal ideal such that $\varphi \in R^{(s+1)\times s}$ is obtained by specializing $\mathcal{M} \in S^{(s+1) \times s}$ via the map $S \twoheadrightarrow S/\mm_\alpha \cong R$.
		From \cite[Lemma 6.2]{SPECIALIZATION_ARON} the generic ideal $I \otimes_A \kk(\mathbf{z}) \subset S(\xi) = \kk(\mathbf{z})[x_0,\ldots,x_r]$ is perfect of height two and satisfies the condition $G_{r+1}$, and so \cite[Theorem 5.7]{MIXED_MULT} implies that the projective degrees of $\FF(\xi) : \PP_{\kk(\mathbf{z})}^r \dashrightarrow \PP_{\kk(\mathbf{z})}^s$ are equal to $d_i(\FF(\xi)) = e_{r-i}(\mu_1,\ldots,\mu_s)$.
		Finally, \autoref{thm_specialization_rat_map} yields that $d_i(\Psi) = d_i^\FF(\mm_\alpha) \le d_i^\FF(\xi) = e_{r-i}(\mu_1,\ldots,\mu_s)$ for all $0 \le i \le r$.
		
		(iii) This part follows verbatim as part (ii), except that we 
		need to use \cite[Lemma 2.12]{MULT_SAT_FIB_GOR_3} and \cite[Theorem 5.8]{MIXED_MULT}.
	\end{proof}
	
	From \autoref{thm_specialization_rat_map}, we have that the projective degrees and the degree of the image of a rational map $\FF : \PP_A^r \dashrightarrow \PP_A^s$ behave as lower semicontinuous functions under specialization; accordingly, these invariants cannot increase under specialization.
	However, it turns out that the degree of a rational map is a much more erratic invariant, and it seems that \autoref{cor_spec_deg_map} is the most general result one can hope for. Indeed, the following two examples show that the degree of a rational map can either increase or decrease under specialization. 
	
	\begin{example}[The degree of a rational map can increase]
		Let $\QQ$ be the field of rational numbers, and $A = \QQ[a]$ and $S = A[x_0,x_1,x_2]$ be polynomial rings.
		Consider the following  matrix 
		$$
		\mathcal{M} = \begin{pmatrix}
			{x}_{0}&{x}_{1}&{x}_{0}^{2}\\
			{x}_{1}&{x}_{0}&{x}_{1}^{2}\\
			{x}_{0}+a\,{x}_{2}&{x}_{1}&{x}_{2}^{2}\\
			0&{x}_{0}&0\end{pmatrix},
		$$ 
		and let $\FF : \PP_A^2 \dashrightarrow \PP_A^3$ be a rational map with base ideal $I = (f_0,f_1,f_2,f_3) \subset S$ given by
		$$
		I = I_3(\mathcal{M}) = {\small\begin{pmatrix}
				{x}_{0}^{2}{x}_{1}^{2}+a\,{x}_{0}{x}_{1}^{2}{x}_{2}-{x}_{0}{x}_{1}{x}_{2}^{2},\\
				-{x}_{0}^{4}-a\,{x}_{0}^{3}{x}_{2}+{x}_{0}^{2}{x}_{2}^{2},\\
				{x}_{0}^{3}{x}_{1}-{x}_{0}^{2}{x}_{1}^{2},\\
				{x}_{0}^{4}-{x}_{0
				}^{2}{x}_{1}^{2}+a\,{x}_{0}^{3}{x}_{2}-a\,{x}_{1}^{3}{x}_{2}-{x}_{0}^{2}{x}_{2}^{2}+{x}_{1}^{2}{x}_{2}^{2}\end{pmatrix}}.
		$$
		Let $\sT = S \otimes_A A[y_0,y_1,y_2,y_3]$ be a standard bigraded polynomial ring over $A$ and write the Rees algebra $\Rees(I)$ as $\Rees(I) \cong \sT/\mathcal{J}$.
		The ideal $\mathcal{J} \subset \sT$ is equal to 
		$$
		\small
		\begin{pmatrix}
			\mathtt{{x}_{1}{y}_{0}+{x}_{0}{y}_{1}+{x}_{1}{y}_{2}+{x}_{0}{y}_{3}},\\
			\mathtt{{x}_{0}{y}_{0}+{x}_{1}{y}_{1}+\left({x}_{0}+a\,{x}_{2}\right){y}_{2}},\\
			\left({x}_{0}{x}_{1}-{x}_{1}^{2}\right){y}_{1}+\left({x}_{0}^{2}+a
			\,{x}_{0}{x}_{2}-{x}_{2}^{2}\right){y}_{2},\\
			\mathtt{{x}_{2}{y}_{0}^{2}+\left(-a\,{x}_{0}+a\,{x}_{1}-{x}_{2}\right){y}_{1}^{2}+\left(a^{2}+2\right){x}_{2}{y}_{0}{y}_{2}-a\,{x}_{1}{y}_{1}{y}_{2}+{x}_{2}{y
				}_{2}^{2}+\left(-a\,{x}_{0}+a\,{x}_{1}-{x}_{2}\right){y}_{1}{y}_{3}},\\
			{y}_{0}^{4}-2\,{y}_{0}^{2}{y}_{1}^{2}+{y}_{1}^{4}+\left(a^{2}+4\right){y}_{0}^{3}{y}_{2}-4\,{y}_{0}{y}_{1}^{2}{y}_{2}+a^{2}{y
			}_{1}^{3}{y}_{2}+\left(2\,a^{2}+6\right){y}_{0}^{2}{y}_{2}^{2}-2\,{y}_{1}^{2}{y}_{2}^{2}+\\\left(a^{2}+4\right){y}_{0}{y}_{2}^{3}+{y}_{2}^{4}-2\,{y}_{0}^{2}{y}_{1}{y}_{3}+2\,{y}_{1}^{3}{y}_{3}-4\,{
				y}_{0}{y}_{1}{y}_{2}{y}_{3}+2\,a^{2}{y}_{1}^{2}{y}_{2}{y}_{3}-2\,{y}_{1}{y}_{2}^{2}{y}_{3}+{y}_{1}^{2}{y}_{3}^{2}+a^{2}{y}_{1}{y}_{2}{y}_{3}^{2}\end{pmatrix}.
		$$
		This can be computed in a computer algebra system like \texttt{Macaulay2} \cite{MACAULAY2}.
		The ideal $\mathcal{J} \subset \sT$ is generated by $5$ bihomogeneous polynomials in $\sT$. 
		The first, second and fourth generators of $\mathcal{J}$ are linear in the variables $x_i$.
		Let $\LL = \QQ(a) = \Quot(A)$ and $\mathbb{G} : \PP_\LL^2 \dashrightarrow \PP_\LL^3$ be the generic rational map with base ideal $I \otimes_A \LL \subset \LL[x_0,x_1,x_2]$.
		From \cite[Theorem 2.18]{AB_INITIO}, we can check that $\mathbb{G}$ is birational, i.e., $\deg(\mathbb{G})=1$.
		Alternatively, we give a short direct argument. 
		We may assume that $\LL$ is algebraically closed.
		By considering the generator $f_2  = x_0^3x_1-x_0^2x_1^2 = x_0^2x_1(x_0-x_1)$ of $I$, we obtain the morphism 
		$
		h : D(f_2) \rightarrow \PP_\LL^3
		$.
		Notice that $D(f_2)$ lies inside the affine patch  $\mathbb{A}_\LL^2 \subset \PP_\LL^2$ with $x_0 = 1$.
		For any point $p = (p_0: p_1:p_2:p_3) \in \PP_\LL^3$ in the image of $h$,  if $(1,\alpha_1,\alpha_2) \in h^{-1}(p)$, then we get the following linear system
		$$
		\begin{array}{cccc}
			p_1 \alpha_1 & + & p_2a\alpha_2 &= -(p_0+p_2)\\
			(p_0+p_2) \alpha_1 && &= -(p_1+p_3) 
		\end{array}
		$$
		that is derived from the two linear syzygies in $\mathcal{M}$; this system has a unique solution $(1,\alpha_1,\alpha_2)$ if the determinant $-ap_2(p_0+p_2)$ is non-zero.
		It follows that for any $\alpha = (1, \alpha_1, \alpha_2) \in D(f_2)\cap D(af_2(f_0+f_2)) = D(f_2(f_0+f_2))$, the fiber $h^{-1}(h(\alpha))$ has one element.
		Thus $\mathbb{G}$ is birational.

		On the other hand, we make the specialization $a = 0$, which gives the matrix 
		$$
		M = \begin{pmatrix}
			{x}_{0}&{x}_{1}&{x}_{0}^{2}\\
			{x}_{1}&{x}_{0}&{x}_{1}^{2}\\
			{x}_{0}&{x}_{1}&{x}_{2}^{2}\\
			0&{x}_{0}&0\end{pmatrix}.
		$$ 
		Let $\mathbb{g} : \PP_\kk^2 \dashrightarrow \PP_\kk^3$ be a rational map with base ideal 
		$$
		J = I_3(M) = \left(-{x}_{0}^{4}+{x}_{0}^{2}{x}_{1}^{2}+{x}_{0}^{2}{x}_{2}^{2}-{x}_{1}^{2}{x}_{2}^{2},\, {x}_{0}^{3}{x}_{1}-{x}_{0}^{2}{x}_{1}^{2},\, {x}_{0}^{4}-{x}_{0}^{2}{x}_{2}^{2},\, {x}_{0}^{2}{x
		}_{1}^{2}-{x}_{0}{x}_{1}{x}_{2}^{2}\right).
		$$
		In this case $\mathbb{g}$ is not birational, indeed $\deg(\mathbb{g}) = 2$.
		
		Therefore, under the above specialization, we obtain $\deg(\mathbb{g}) = 2 > 1 = \deg(\mathbb{G})$.
	\end{example}

	\begin{example}[{The degree of a rational map usually decreases}]
		We recall an example \cite[Example 6.5]{SPECIALIZATION_ARON} where the degree of a rational map can decrease arbitrarily under specialization.
		Let $m \ge 1$ be an integer.
		Let $\kk$ be a field, and $A=\kk[a]$ and $S=A[x_0,x_1,x_2]$ be polynomial rings. Consider the matrix 
		$$
		\mathcal{M} = 
		\begin{pmatrix}
			x & zy^{m-1}\\
			-y & zx^{m-1} + y^m\\
			az & zx^{m-1}
		\end{pmatrix}
		$$
		with entries in $S$.
		Let $\FF : \PP_A^2 \dashrightarrow \PP_A^2$ be a rational map with base ideal $I = I_2(\mathcal{M}) \subset S$.
		Let $\LL = \kk(a) = \Quot(A)$ and $\mathbb{G} : \PP_\LL^2 \dashrightarrow \PP_\LL^2$ be the generic rational map with base ideal $I \otimes_A \LL \subset \LL[x_0,x_1,x_2]$.
		For any $\beta \in \kk$, let $\nnn_\beta =(a-\beta) \in A$ and $\mathbb{g}_\beta : \PP_\kk^2 \dashrightarrow \PP_\kk^2$ be a rational map with base ideal $I(\nnn_\beta) \subset \kk[x_0,x_1,x_2]$.
		Then, we have that $\deg(\mathbb{G}) = m$ and 
		$$
		\deg(\mathbb{g}_\beta) = \begin{cases}
			1 & \text{ if } \beta = 0\\
			m & \text{ if } \beta \neq 0.
		\end{cases}
		$$
		So, the specialization $a = 0$ gives an arbitrary decrease in degree $\deg(\mathbb{g}_0) = 1 < m = \deg(\mathbb{G})$.
	\end{example}

	\section{Polar multiplicities, Segre numbers and new set of invariants}
	\label{sect_invariants_cycles}
	
	In this section, we prove several results regarding polar multiplicities and Segre numbers of an ideal, and we introduce a new related invariant.
	These invariants are defined as a special case of the general notion of polar multiplicities due to Kleiman and Thorup \cite{KLEIMAN_THORUP_GEOM, KLEIMAN_THORUP_MIXED}.
	Here, an important goal for us is to extend several of the results of Gaffney and Gassler \cite{GG} from their analytic setting to an algebraic one over a Noetherian local ring.
	The following setup is used throughout this section.
	
	\begin{setup}
		\label{setup_polar_Segre}
		Let $(R, \mm, \kappa)$ be a Noetherian local ring with maximal ideal $\mm$ and residue field $\kappa$.
		Let  $d := \dim(R)$ and $X := \Spec(R)$.
		Let $I \subset R$ be a proper  ideal  generated by elements $f_1,\ldots,f_m \in  R$.
		We consider the Rees algebra $B := \Rees(I) := R[IT] = \bigoplus_{v \ge 0} I^vT^v \subset R[T]$ of the ideal $I$. 
		We have a natural homogeneous presentation $W := R[y_1,\ldots,y_m] \surjects \Rees(I)$, $y_i \mapsto f_iT$, where $W$ is a standard graded polynomial ring over $R$.
		Let $P := {\rm Bl}_I(X) = \Proj(B) \subset \Proj(W) = \PP_R^{m-1}$ be the blowup of $X$ along $I$ and consider the natural projection 
		$$
		\pi : P \subset \PP_R^{m-1} \rightarrow X.
		$$
		Let $E := \pi^{-1}(V(I)) \cong \Proj(G) \subset P$ be the exceptional divisor and $G:=\gr_I(R) := \bigoplus_{v \ge 0} I^v/I^{v+1}$ be the corresponding associated graded ring.  
		The blowup $P$ has a natural  affine open cover $P = \bigcup_{i=1}^m U_i$ where $U_i = \Spec\left(\big[B_{y_i}\big]_0\right)$. 
		More precisely, we can write
		$$
		U_i= \Spec\left(R[I/f_i]\right),
		$$
		where $R[I/f_i]$ denotes the $R$-subalgebra of $R_{f_i}$ generated by all $f/f_i$ with $f \in I$.
		Notice that the local equation of $E$ on $U_i$ is given by $f_i \in R[I/f_i]$.
	\end{setup}
	
	We now briefly recall the general notion of polar multiplicities due to Kleiman and Thorup \cite{KLEIMAN_THORUP_GEOM, KLEIMAN_THORUP_MIXED}.
	Here we shall freely use the results from the references \cite{KLEIMAN_THORUP_GEOM, KLEIMAN_THORUP_MIXED, cidruiz2024polar} regarding polar multiplicities.
	Let $M$ be a finitely generated graded $B$-module and $\FF = \widetilde{M}$ the corresponding coherent $\OO_P$-module.
	The function $(v, n) \mapsto \length_R\left(M_v/\mm^{n+1}M_v\right)$ eventually coincides with a bivariate polynomial $P_M(v, n)$ of degree equal to $\dim(\Supp(\FF))$.
	Then, for all $r \ge \dim(\Supp(\FF))$, we can write 
	$$
	P_M(v, n) \;=\;  \sum_{i=0}^r\; \frac{m_r^i(M)}{i!\,(r-i)!}\;v^{r-i}n^i \;+\; \text{(lower degree terms)}.
	$$
	We say that the invariants $m_r^{i}(M)$ are the \emph{polar multiplicities} of $M$.
	Recall that $\dim(P) \le d$ and $\dim(E) \le d-1$.
	Our main interest is on the following invariants:

	\begin{definition}
		\begin{enumerate}[\rm (i)]
			\item For all $0 \le i \le d$, we say that $m_i(I, R) := m_{d}^{d-i}(B)$ is the \emph{$i$-th polar multiplicity} of the ideal $I \subset R$.
			\item For all $1 \le i \le d$, we say that $c_i(I, R) := m_{d-1}^{d-i}(G)$ is the \emph{$i$-th Segre number} of the ideal $I \subset R$.		
			\item For all $1 \le i \le d$, we say that $\nu_i(I, R) := m_i(I, R) + c_i(I, R)$ is the \emph{$i$-th polar-Segre multiplicity} of the ideal $I \subset R$.
			By convention, we also set $\nu_0(I, R) =m_0(I, R)$.
		\end{enumerate}
	\end{definition}
	
	By \cite[Proposition 2.10]{cidruiz2024polar}, we get $m_d(I, R) = m_{d}^0(B) = j_{d+1}(B)$, and since $\dim(B/\mm B) \le d$, it follows that $m_d(I, R) = j_{d+1}(B) = 0$ (see \cite[\S 6.1]{FLENNER_O_CARROLL_VOGEL}).
	For all $0 \le i \le d-1$, the polar multiplicity $m_i(I, R)$ is also referred to as the \emph{mixed multiplicity} $e_i(\mm\mid I)$ (see, e.g., \cite{Trung2001}).

	We shall need some very basic rudiments from intersection theory. 
	Since we are working over our Noetherian local ring $R$ (and not over a field), the usual developments from Fulton's book \cite{FULTON_INTERSECTION_THEORY} do not suffice. 
	In terms of a suitable \emph{dimension function}, we could use available extensions of intersection theory (see, e.g., \cite[Chapter 20]{FULTON_INTERSECTION_THEORY}, \cite[\href{https://stacks.math.columbia.edu/tag/02P3}{Chapter 02P3}]{stacks-project}, \cite{THORUP}).
	However, as we shall not require a notion of rational equivalence, we present our results in terms of cycles and quickly develop the necessary concepts. 
	
	\begin{notation}
		\begin{enumerate}[\rm (i)]
			\item 	Let $Y$ be a Noetherian scheme. 
			We denote by $Z_k(Y)$ the free group of $k$-dimensional cycles. 
			For a coherent sheaf $\FF$ on $Y$ and an integer $k \ge \dim(\Supp(\FF))$, we denote by $\big[\FF\big]_k \in Z_k(Y)$ the associated $k$-cycle. 
			For a closed subscheme $Z \subset Y$ and an integer $k \ge \dim(Z)$, we denote by $\big[Z\big]_k = \big[\OO_{Z}\big]_k \in Z_k(Y)$ the associated $k$-cycle.	
			\item Given $k$-cycle $\xi = \sum_i l_i[R/\pp_i] \in Z_k(X)$, its multiplicity is given by $e_{k}(\xi) := \sum_i l_ie_k(R/\pp_i)$.
		\end{enumerate}
	\end{notation}
	
	Below we give a short self-contained result regarding the push-forward of cycles along a projection (cf., \cite[\href{https://stacks.math.columbia.edu/tag/02R6}{Lemma 02R6}]{stacks-project}, \cite[Proposition 4.3]{THORUP}).
	
	\begin{defprop}
		\label{defprop_push}
		Let $S$ be a standard graded $R$-algebra and consider the projective morphism $\eta: Y = \Proj(S) \rightarrow X = \Spec(R)$.
		Given an integral closed subscheme $Z \subset Y$, the push-forward is defined as $$
		\eta_*\left([Z]\right) \;:=\; \begin{cases}
			\big[K(Z) : K(Z')\big] \cdot [Z'] & \quad \text{ if } \dim(Z) = \dim(Z') \\
			0 & \quad \text{ otherwise},
		\end{cases}
		$$
		where $Z'=\eta(Z)$, and $K(Z)$ and $K(Z')$ denote the function fields of $Z$ and $Z'$, respectively.
		The push-forward map $\eta_* : Z_k(Y) \rightarrow Z_k(X)$ is then determined by linearity.
		Let $\FF$ be a coherent $\OO_Y$-module and $k \ge \dim(\Supp(\FF))$.
		Then 
		$
		\eta_*\left(\big[\FF\big]_k\right) = \big[\eta_*(\FF)\big]_k \in Z_k(X).
		$
	\end{defprop}
	\begin{proof}
		Let $Z = \Proj(S/\fP)$ where $\fP \subset S$ is a relevant prime ideal, and set $Z' = \eta(Z) = \Spec(R/\pp)$ where $\pp = \fP \cap R$.
		From \cite[Lemma 1.2.2]{vasconcelos1994arithmetic}, we get $\dim(S/\fP) = \dim(R/\pp) + \trdeg_{R/\pp}(S/\fP)$.
		Since we have $\trdeg_{R/\pp}(S/\fP) \ge 1$, it follows that $\dim(Z) \ge \dim(Z')$.
		Let $M$ be a finitely generated graded $S$-module with $\FF \cong \widetilde{M}$ and $\HH_{S_+}^0(M)=0$.
		Notice that we may substitute $S$ by $S/\Ann_S(M)$ and $R$ by $R/(R\cap \Ann_S(M))$.
		Therefore we assume that $k \ge \dim(Y)$ and $k\ge\dim(R)$, and that any minimal prime of $S$ is relevant.
		
		Let $\pp \subset R$ be a minimal prime of $N=\HH^0(X, \eta_*(\FF))$ of dimension $k$.
		Notice that any relevant prime of $S$ contracting to $\pp$ should be minimal. 
		Thus the fiber $\eta^{-1}(\pp)$ is finite, and so we can find an affine open neighborhood $V \subset X$ of $\pp$ such that $\eta^{-1}(V) \rightarrow V$ is finite (see \cite[\href{https://stacks.math.columbia.edu/tag/02NW}{Lemma 02NW}]{stacks-project}, \cite[Exercise II.3.7]{HARTSHORNE}).
		We set $\eta^{-1}(V) = \Spec(A)$ and choose a finitely generated $A$-module $L$ such that $\widetilde{L} \cong {\FF\mid}_{\eta^{-1}(V)}$.
		Since $R_\pp$ is an Artinian local ring and $A_\pp = A \otimes_R R_\pp$ is module-finite over $R_\pp$, it follows that $A_\pp$ is Artinian.
		We have the equality 
		$$
		\sum_{\qqq}[\kappa(\qqq):\kappa(\pp)]\cdot \length_{A_\qqq}(L_\qqq)\;=\; \length_{R_\pp}(N_\pp),
		$$ 
		where the sum runs through the minimal primes of $A$ contracting to $\pp$.
		Finally, notice that the right-hand side of the above equality is the coefficient in $[\eta_*(\FF)]_k$ corresponding to $\pp$ and that the left-hand side is the coefficient in $\eta_*([\FF]_k)$ corresponding to $\pp$.
	\end{proof}
	
	\begin{notation}
		\label{nota_gen_hyper}
		Given elements $a_1,\ldots,a_m$ in $R$, say that $g=a_1f_1+\cdots+a_mf_m$ is the associated element in $I$, that $\ell = a_1y_1+\cdots+a_my_m \in B_1$ is the associated linear form, that $D = V(g) \subset X$ is the associated hypersurface in $X$, and that $H = V_+(\ell) \subset P$ is the associated hyperplane in $P$.
		We denote by $\pi^*D := V_+(gB) = \Proj(B/gB)$ the pullback of the hypersurface $D$.
		When the residue field $\kappa$ is infinite, we use the following conventions:
		\begin{itemize}[--]
			\item We say that $D \subset X$ is general (equivalently $H \subset P$ is general)  if $g$ is a general element in $I$ (equivalently $\ell$ is a general element in $B_+$).
			\item We say that a sequence of hyperplanes $\underline{H} = H_1,\ldots,H_k$ is \emph{a  sequence of general hyperplanes} in $P$ if the associated sequence $\underline{g}=g_1,\ldots,g_k$ is a sequence of general elements in $I$.
		\end{itemize}
	\end{notation}		
	
	When $\kappa$ is infinite and $g \in I$ is a general element, the following remark shows that $\pi^*D$ is an effective Cartier divisor on $P$ even if $D=V(g)$ is not a divisor on $X$.
	
	\begin{remark}
		\label{rem_pullback_g}
		($\kappa$ infinite).
		\label{rem_effective_Div_pullback}
		Let $\overline{R} = R/(0:_RI^\infty)$ and $\overline{X} = \Spec(\overline{R})$.
		Since ${\rm Bl}_I(X) \cong {\rm Bl}_I(\overline{X})$, we may assume that $(0:_RI^\infty)=0$, and so by prime avoidance we get that $g$ is a nonzerodivisor when $g \in I$ is general.
		This shows that $\pi^*D$ is an effective Cartier divisor when $g \in I$ is general.
	\end{remark}

	The next notation includes an inequality that will be useful in our approach.
	This inequality is related to when a general element of $I$ is a $\mathscr{G}$-parameter on the Rees algebra $B=\Rees(I)$ (in the sense of \cite[Definition 2.7]{cidruiz2024polar}).
	
	\begin{notation}
		\label{nota_G_param}
		We say that the \emph{order} of the ideal $I$ is given by $o(I) := \sup\lbrace \beta \in \ZZ_{\ge0} \mid I \subseteq \mm^\beta \rbrace$.
		Let $\delta = o(I)$.
		Consider the standard bigraded algebra $\mathscr{G} := \gr_\mm(B)$ with bigraded parts $[\sG]_{(v, n)} = \mm^nB_v/\mm^{n+1}B_v$. 
		Take  the $\kappa$-vector subspace $\bb = I/\mm^{\delta+1} \subset \mm^\delta/\mm^{\delta+1} = [\sG]_{(0,\delta)}$.
		Let 
		$$
		\iniTerm^\delta(I) \;:=\; \bb \cdot \sG \;=\; \left(\iniTerm(f) \mid f \in I \text{ and } o(f) = \delta\right) \subset \sG
		$$ 
		be the ideal generated by the initial forms of elements in $I$ of order $\delta$.
		If the following strict inequality 
		$$
		\dim\left(\biProj\left(\sG/\iniTerm^\delta(I)\right)\right) \;<\; d-1
		$$
		holds, we say that $I$ satisfies the \emph{$\sG$-parameter condition generically}.
	\end{notation}
	
	The following proposition is inspired by one of the technical steps in Fulton's proof of the commutativity of intersecting with Cartier divisors: indeed, similarly to \cite[Lemma 2.4]{FULTON_INTERSECTION_THEORY}, we express the pullback  $\pi^*D$ as the sum of the exceptional divisor $E$ and a hyperplane $H$ in the blowup $P = {\rm Bl}_I(X)$.
	As a consequence, we get inequalities relating the polar multiplicities and the Segre numbers of $I$.

	\begin{proposition}
		\label{prop_decomp_div}
		{\rm(}$\kappa$ infinite{\rm)}.
		Let $H \subset P$ be a general hyperplane, $D = V(g)$ be the associated hypersurface and $\ell \in B_+$ be the associated linear form. 	 
		Then the following statements hold: 
		\begin{enumerate}[\rm (i)]
			\item $\pi^*D$ and $H$ are effective Cartier divisors.
			\item $\pi^*D = E + H$.
			\item Let $\delta = o(I)$. 
			Then we have the inequality 
			$$
			\delta \cdot m_{i-1}(I, R) \;\le\; m_i(I, R) + c_{i}(I, R) \;=\; \nu_i(I, R),
			$$	
			and equality holds for all $1 \le i \le d$ if and only if $I$ satisfies the $\sG$-parameter condition generically {\rm(}see \autoref{nota_G_param}{\rm)}.
		\end{enumerate}
	\end{proposition}
	\begin{proof}
		(i) From \autoref{rem_effective_Div_pullback}, we get that $\pi^*D$ is an effective Cartier divisor. 
		We have that $H$ is also an effective Cartier divisor by prime avoidance.
		
		(ii) We write $g = a_1f_1+\cdots+a_mf_m$ and $\ell = a_1y_1+\cdots + a_my_m$.
		Consider the affine open subscheme $U_i = \Spec(R[I/f_i])$. 
		As we mentioned before, the local equation of $E$ on $U_i$ is given by $f_i \in R[I/f_i]$.
		On the other hand, the local equations of $\pi^*D$ and $H$ on $U_i$ are given by 
		$$
		a_1f_1 + \cdots + a_mf_m \;\in\; R[I/f_i] \qquad \text{ and } \qquad  \frac{a_1f_1 + \cdots + a_mf_m}{f_i} \;\in\; R[I/f_i],
		$$
		respectively. 
		This shows the equality $\pi^*D = E + H$ on each $U_i$, and so the equality holds globally on the whole  blowup $P = {\rm Bl}_I(X)$.
		
		(iii)	
		First, we check that in the vacuous case $\dim(P) < d$, we have that $m_i(I, R) = 0$, $\nu_i(I, R) =0$, $\dim(B) \le d$ and $\dim\left(\biProj(\sG)\right) \le d-2$ (hence the equivalence statement holds trivially).
		Therefore, we assume $\dim(P)=d$.

		As $g \in I$ is general, we may assume that $g \in \mm^\delta \setminus \mm^{\delta+1}$.
		Notice that $\big[\pi^*D\big]_{d-1} = \big[\widetilde{B/gB}\big]_{d-1}$ and $\big[H\big]_{d-1} = \big[\widetilde{B/\ell B}\big]_{d-1}$, and thus part (ii) yields the equality of cycles 
		$$
		\big[\widetilde{B/gB}\big]_{d-1} \;=\; \big[\widetilde{B/\ell B}\big]_{d-1} + \big[\widetilde{G}\big]_{d-1} \;\in\; Z_{d-1}(P);
		$$
		see, e.g., \cite[Lemma A.2.5]{FULTON_INTERSECTION_THEORY}.
		Hence the additivity of polar multiplicities (see \cite[Corollary 2.6]{cidruiz2024polar}) implies that 
		$$
		m_{d-1}^i(B/gB) \;=\; m_{d-1}^i(B/\ell B) + m_{d-1}^i(G).
		$$
		Due to \cite[Theorem 2.8]{cidruiz2024polar}, we have $m_{d-1}^i(B/gB) \ge \delta \cdot m_d^{i+1}(B)$ and an equality holds for all $0 \le i \le d-1$ if and only if $g$ is a $\mathscr{G}$-parameter on $B$ (in the sense of \cite[Definition 2.7]{cidruiz2024polar}).
		Since $\ell \in B_+$ is a general element, by utilizing  \cite[Proposition 2.10]{cidruiz2024polar},  we obtain $m_{d-1}^i(B/\ell B) = m_d^i(B)$.
		Finally, by combining everything we get the inequality 
		$$
		\delta \cdot m_{i-1}(I, R) \;=\;  \delta \cdot m_d^{d-i+1}(B) \;\le\; m_d^{d-i}(B) + m_{d-1}^{d-i}(G) \;=\; m_i(I, R) + c_i(I, R), 
		$$
		and an equality holds for all $1 \le i \le d$ if and only if $g$ is a $\mathscr{G}$-parameter on $B$.
		Since $g$ is a general element of $I$, it follows that $g$ is a $\sG$-parameter on $B$ if and only if $I$ satisfies the $\sG$-parameter condition generically.
	\end{proof}

	\begin{remark}
		The equality of \autoref{prop_decomp_div}(ii) can also be derived as follows. 
		Consider the extended Rees algebra $\Rees^+(I) := R[IT, T^{-1}]$.
		Notice that the equality $g = \ell \cdot T^{-1}$ holds in $\Rees^+(I)$ and recall the isomorphism $\gr_I(R) \cong \Rees^+(I)/ T^{-1}\Rees^+(I)$.
	\end{remark}

	Next, we introduce the notion of polar schemes and Segre cycles. 
	We also introduce a new type of cycle that will be fundamental in our treatment. 
	The following additional data is fixed for the rest of the section.

	\begin{setup}
		\label{setup_segre_cycles}
		Assume \autoref{setup_polar_Segre} and that the residue field $\kappa$ is infinite.
		Let $\underline{H} = H_1,\ldots, H_d$ be a sequence of general hyperplanes, and denote by $\underline{g}=g_1,\ldots,g_d$ the associated sequence of elements in $I$ and by $\underline{\ell} = \ell_1, \ldots, \ell_d$ the associated sequence of linear forms in $B_+$.
		We also set $D_i = V(g_i) \subset X$ and recall that the pullback $\pi^*D_i$ is an effective Cartier divisor on $P$ (see \autoref{rem_effective_Div_pullback}).
		We introduce the following objects: 
		\begin{enumerate}[\rm (i)]
			\item For $1 \le i \le d$, we say that the \emph{$i$-th polar scheme} (with respect to $\underline{H}$) is given by the following schematic-image
			$$
			P_i(I, X)  \;=\; P_i^{\underline{H}}(I, X)  \;:=\; \pi\left(H_1 \cap \cdots \cap H_i\right).
			$$
			\item For $1 \le i \le d$, we say that the \emph{$i$-th Segre cycle} (with respect to $\underline{H}$) is given by 
			$$
			\Lambda_i(I, X)  \;=\; \Lambda_i^{\underline{H}}(I, X)  \;:=\; \pi_*\left(\big[E \cap H_1 \cap \cdots \cap H_{i-1}\big]_{d-i}\right) \;\in\; Z_{d-i}(X).
			$$
			\item For $1 \le i \le d$, we say that the \emph{$i$-th polar-Segre cycle} (with respect to \underline{$H$}) is given by 
			$$
			V_i(I, X)  \;=\; V_i^{\underline{H}}(I, X)  \;:=\; \pi_*\left(\big[H_1 \cap \cdots \cap H_{i-1} \cap \pi^*D_i\big]_{d-i}\right) \;\in\; Z_{d-i}(X).
			$$
		\end{enumerate} 
		By convention, we set $P_0(I, X):=\pi(P)$ and $V_0(I, X) := \left[P_0(I, X)\right]_d \in Z_d(X)$.
	\end{setup}
	
	\begin{remark}
		\label{rem_filter_reg_seq}
		By prime avoidance, we can assume that $g_1,\ldots,g_{i}$ is a regular sequence on $R_\pp$ for each $\pp \in V(g_1,\ldots,g_i) \setminus V(I)$ (we say that $g_1,\ldots,g_i$ are a filter-regular sequence with respect to $I$; see \cite[Appendix]{STUCKRAD_VOGEL_BUCHSBAUM_RINGS}).
		Therefore, we have that $(g_1,\ldots,g_{i}):_RI^\infty$ either equals $R$ or it has height $i$, and so we obtain $\dim(R/(g_1,\ldots,g_{i}):_RI^\infty) \le d-i$.
		Similarly, we get $\dim(R/(g_1,\ldots,g_{i-1}):_RI^\infty+g_iR) \le d-i$.
	\end{remark}

	The theorem below gives an important description of the invariants $m_i(I, R)$, $c_i(I, R)$ and $\nu_i(I, R)$.
	It shows that these invariants are naturally the multiplicities of the cycles introduced in \autoref{setup_segre_cycles}.
	
	\begin{theorem}
		\label{thm_polar_segre_cycles}
		Assume \autoref{setup_segre_cycles}.
		Then the following statements hold: 
		\begin{enumerate}[\rm (i)]
			\item For all $0 \le i \le d$, we have the equalities $m_i(I, R) = e_{d-i}\left(P_i(I, X)\right)$ and 
			$$
			P_i(I, X) \;=\; \Spec\left(R/(g_1,\ldots,g_i):_RI^\infty\right).
			$$
			\item For all $1 \le i \le d$, we have the equalities $c_i(I, R) = e_{d-i}\left(\Lambda_ i(I, X)\right)$ and 
			$$
			\Lambda_ i(I, X) \;=\; \sum_{\substack{\pp \in V\left((g_1,\ldots,g_{i-1}):_RI^\infty\right)\\ \pp \in V(I), \; \dim(R/\pp)=d-i}} e\big(I,\, R_\pp/(g_1,\ldots,g_{i-1})R_\pp:_{R_\pp} I^\infty R_\pp\big)\cdot \left[R/\pp\right] \;\in\; Z_{d-i}(X). 
			$$
			\item For all $1 \le i \le d$, we have the equalities $\nu_i(I, R) = e_{d-i}(V_i(I, X))$ and 
			$$
			V_i(I, X) \;=\; \left[P_i(I, X)\right]_{d-i} + \Lambda_i(I, X) \;=\; \big[\Spec\left(R/(g_1,\ldots,g_{i-1}):_RI^\infty+g_iR\right)\big]_{d-i} \;\in\, Z_{d-i}(X).
			$$		
		\end{enumerate}		
	\end{theorem}
	\begin{proof}
		Let $\aaa_i:= (g_1,\ldots,g_i) \subset R$, $R_i := R/\aaa_i$, $X_i := \Spec(R_i)$ and $\overline{X}_i :=\Spec(R/\aaa_i:_RI^\infty)$.
		
		(i) 
		First, set $i=0$.
		By \cite[Proposition 2.10]{cidruiz2024polar}, we have $m_0(I, R) = m_d^d(B) = e_d\left(\HH^0(P, \OO_P)\right)$.
		We have that $P_0(I, X) = \pi(P) = \Spec(R/\aaa)$ where $\aaa \subset R$ is the kernel of the natural map $R \xrightarrow{\;\rm nat\;} \HH^0(P, \OO_P)$ (see, e.g., \cite[Proposition 10.30]{GORTZ_WEDHORN}, \cite[Exercise II.3.11]{HARTSHORNE}).  
		On the other hand, we have a four-term exact sequence 
		$$
		0 \;\longrightarrow\; \big[\HH_{B_+}^0(B)\big]_0  \;\longrightarrow\; R = B_0 \;\xrightarrow{\; \rm nat \;}\; \HH^0(P, \OO_P) \;\longrightarrow\; \big[\HH_{B_+}^1(B)\big]_0 \;\longrightarrow\; 0.
		$$
		Since $B = R[IT]$ is the Rees algebra of $I$, it follows that $\big[\HH_{B_+}^0(B)\big]_0 = 0 :_R I^\infty$.
		Hence we have $\aaa = 0 :_R I^\infty$ and $P_0(I, X) = \pi(P) = \Spec(R/0:_RI^\infty)$.
		Moreover, we get the short exact sequence 
		$$
		0 \rightarrow R/0:_RI^\infty \rightarrow \HH^0(P, \OO_P) \rightarrow \big[\HH_{B_+}^1(B)\big]_0 \rightarrow 0,
		$$
		and so to prove $m_0(I, R) = e_d(P_0(I, X))$, it suffices to show that $\dim\big(\big[\HH_{B_+}^1(B)\big]_0\big) < d$.
		Let $\overline{B} := B / \HH_{B_+}^0(B)$. 
		As $\HH_{B_+}^1(B) \cong \HH_{B_+}^1(\overline{B})$ and $\left[\overline{B}\right]_0 = R/0:_RI^\infty$, we have that $\HH_{B_+}^1(B) \otimes_R R_\pp = 0$ for any minimal prime $\pp \in {\rm Min}(R)$ that contains $I$.
		However, if a prime $\pp \in \Spec(R)$ does not contain $I$, we obtain 
		$$
		\big[\HH_{B_+}^1(B)\big]_0 \otimes_R R_\pp \;\cong\; \big[\HH_{B_+}^1(B \otimes_R R_\pp)\big]_0 \;\cong\; \big[\HH_T^1\left(R_\pp[T]\right)\big]_0 = 0.
		$$
		This settles the claim that $\dim\big(\big[\HH_{B_+}^1(B)\big]_0\big) < d$, and so the proof is complete for the case $i=0$.
		
		Since $\ell_1,\ldots,\ell_i$ is a sequence of general elements in $B_+$, \cite[Proposition 2.10]{cidruiz2024polar} yields the equalities $m_i(I, R) = m_d^{d-i}(B) = m_{d-i}^{d-i}(B/(\ell_1,\ldots,\ell_i)B) = e_{d-i}\left(\HH^0(P, \OO_{H_1\cap\cdots\cap H_i})\right)$.
		Consider the natural specialization map 
		$$
		\mathfrak{s} \;:\;  B/(\ell_1,\ldots,\ell_i)B \;\cong\; \bigoplus_{v\ge0} I^v/\aaa_iI^{v-1} \;\;\;\surjects\;\;\; \Rees_{R_i}(IR_i) \;\cong\; \bigoplus_{v\ge0} I^v/(I^v \cap \aaa_i).
		$$
		Since $g_1,\ldots,g_i$ is a sequence of general elements in $I$, we may assume that they form a superficial sequence for $I$ (see \cite[Proposition 8.5.7]{huneke2006integral}), and so \cite[Lemma 8.5.11]{huneke2006integral} yields the equality $\aaa_i I^{v-1} = I^v \cap \aaa_i$ for $v \gg 0$.
		Therefore, we get the following equality (as schemes) 
		\begin{equation}
			\label{eq_specialization}
			H_1\cap \cdots \cap H_i \;=\; \Proj(B/(\ell_1,\ldots,\ell_i)B) \;=\; \Proj(\Rees_{R_i}(IR_i)) \;=\; {\rm Bl}_I(X_i) \;=:\; P_i.
		\end{equation}
		Since we already dealt with the initial case, by substituting $X$ by $X_i$ and setting $i = 0$,  we obtain 
		$$
		P_i(I, X) \;=\; \pi(H_1\cap \cdots \cap H_i) \;=\; \pi(P_i) \;=\; P_0(I, X_i) \;=\; \Spec\left(R/(g_1,\ldots,g_i):_RI^\infty\right)
		$$
		and 
		$$
		m_i(I, R) \;=\; e_{d-i}\left(\HH^0(P, \OO_{H_1\cap\cdots\cap H_i})\right) \;=\; e_{d-i}\left(\HH^0\left(P_i, \OO_{P_i}\right)\right) \;=\; e_{d-i}\left(P_0(I, X_i)\right).
		$$
		This completes the proof of part (i).
		
		(ii) First, set $i=1$.
		By \cite[Proposition 2.10]{cidruiz2024polar}, we get $c_1(I, R) = m_{d-1}^{d-1}(G) = e_{d-1}\left(\HH^0(E, \OO_E)\right)$.
		From \autoref{defprop_push}, we have the equality of cycles
		$$
		\Lambda_1(I, X) \;=\; \big[\pi_*(\OO_E)\big]_{d-1} \;=\; \big[\HH^0(X, \pi_*(\OO_E))\big]_{d-1} \;=\;  \big[\HH^0(E, \OO_E)\big]_{d-1} \;\in\; Z_{d-1}(X),
		$$
		and so it follows that $c_1(I, R) = e_{d-1}(\Lambda_1(I, X))$.
		Since we proved $\pi(P) = \Spec(R/0:_RI^\infty)$ in part (i), it follows that $\Supp\left(\HH^0(E, \OO_E)\right) \subset V(I, 0:_RI^\infty)$.
		Then we get the following equality
		$$
		\left[\HH^0(E, \OO_E)\right]_{d-1} \;=\; \sum_{\substack{\pp \in V(I, 0:_RI^\infty)\\ \dim(R/\pp)=d-1}} \length_{R_\pp}\left(\HH^0(E, \OO_E)\otimes_RR_\pp\right) \cdot [R/\pp] \;\in\; Z_{d-1}(X).
		$$
		Fix a prime $\pp$ in the above summation.
		Let $E_\pp := \Proj\big(\gr_{IR_\pp}(R_\pp)\big)$ be the exceptional divisor of $\Spec(R_\pp)$ along $IR_\pp$.
		Notice that $\dim(R_\pp) = 1$ and $\HH^0(E, \OO_E)\otimes_RR_\pp \cong \HH^0(E_\pp, \OO_{E_\pp})$.
		As a consequence, for $v \gg 0$, we obtain
		$$
		\length_{R_\pp}\left(\HH^0(E_\pp, \OO_{E_\pp})\right) \;=\; \length_{R_\pp}\left(\HH^0(E_\pp, \OO_{E_\pp}(v))\right) \;=\; \length_{R_\pp}\left(I^vR_\pp/I^{v+1}R_\pp\right);
		$$
		the first equality holds for any $v \in \ZZ$ because $\dim(E_\pp) \le 0$. 
		On the other hand, we have 
		$$
		e\left(I, R_\pp/0:_{R_\pp} I^\infty R_\pp\right) \;=\; \length_{R_\pp}\left(I^vR_\pp/I^{v+1}R_\pp\right) \quad \text{ for } \quad v \gg 0.
		$$
		This completes the proof for the case $i = 1$.
		
		Since $\ell_1,\ldots,\ell_{i-1}$ is a sequence of general elements in $B_+$, \cite[Proposition 2.10]{cidruiz2024polar} gives $c_i(I, R) = m_{d-1}^{d-i}(G) = m_{d-i}^{d-i}(G/(\ell_1,\ldots,\ell_{i-1})G) = e_{d-i}\left(\HH^0(P, \OO_{E \cap H_1\cap\cdots\cap H_{i-1}})\right)$.
		From \autoref{eq_specialization}, we obtain the following equality (as schemes)
		\begin{equation}
			\label{eq_specialization_E}
			E \cap H_1 \cap \cdots \cap H_{i-1} \;=\; (H_1\cap \cdots \cap H_{i-1}) \times_X \Spec(R/I) \;=\; P_{i-1} \times_X \Spec(R/I) \;=\;  E_{i-1}
		\end{equation}
		where $E_{i-1} = \Proj(\gr_{IR_{i-1}}(R_{i-1}))$ is the exceptional divisor of $X_{i-1}$ along $IR_{i-1}$.	
		Then the result of part (ii) follows. 
		
		(iii) 
		We have the following equalities 
		\begin{alignat*}{3}
			V_i(I, X) &\;=\; \pi_*\left(\left[H_1 \cap \cdots \cap H_{i-1} \cap \pi^*D_i\right]_{d-i}\right) & \text{by definition of $V_i(I, X)$}\\ 
			&\;=\; \pi_*\left(\left[P_{i-1} \cap \pi^*D_i\right]_{d-i}\right) & \text{by \autoref{eq_specialization}}\\
			&\;=\; \left[V_1(I, \overline{X}_{i-1})\right]_{d-i}  & \text{by definition of $V_1(I, \overline{X}_{i-1})$}\\
			&\;=\; \left[P_1(I, \overline{X}_{i-1})\right]_{d-i} + \left[\Lambda_1(I, \overline{X}_{i-1})\right]_{d-i} & \text{ \quad\quad\; by applying \autoref{prop_decomp_div} to $\overline{X}_{i-1}$}\\
			&\;=\; \left[P_i(I, X)\right]_{d-i} + \Lambda_i(I, X) & \text{by the formulas of part (i) and (ii).}
		\end{alignat*}
		Therefore, to complete the proof of part (iii), we substitute $X$ by $\overline{X}_{i-1}$ and assume that $i=1$.
		In particular, we may assume that $(0:_RI^\infty)=0$ and that $g_1$ is a nonzerodivisor. 
		Set $g := g_1$ and $D := D_1$.	
		By \autoref{defprop_push}, we obtain
		$$
		V_1(I,X) = \big[\pi_*\left(\OO_{\pi^*D}\right)\big]_{d-1} = \sum_{\substack{\pp \in V(g)\\ \dim(R/\pp)=d-1}} \length_{R_\pp}\left(\HH^0(P, \OO_P/g\OO_P)\otimes_RR_\pp\right) \cdot [R/\pp] \;\in\; Z_{d-1}(X).
		$$
		Let $\pp$ be a prime in the above summation, and let $P_\pp := \Proj(B \otimes_R R_\pp)$ the blowup of $\Spec(R_\pp)$ along $IR_\pp$.
		Then the coefficient corresponding to $\pp$ is given by 
		$$
		\length_{R_\pp}\left(\HH^0(P_\pp, \OO_{P_\pp}/g\OO_{P_\pp})\right) \;=\; \length_{R_\pp}\left(\HH^0(P_\pp, \OO_{P_\pp}/g\OO_{P_\pp}(v))\right) \;=\; \length_{R_\pp}\left(I^vR_\pp/gI^{v}R_\pp\right)
		$$
		for $v \gg 0$.
		Expressing length as a multiplicity (see \cite[Corollary 1.2.14]{FLENNER_O_CARROLL_VOGEL}) and utilizing the additivity of multiplicities give the following 
		$$
		\length_{R_\pp}\left(I^vR_\pp/gI^{v}R_\pp\right) \;=\; e\left((g), I^vR_\pp\right) \;=\; e\left((g), R_\pp\right) \;=\; \length_{R_\pp}\left(R_\pp/gR_\pp\right).
		$$
		Hence the required equality $V_1(I, X) = \left[\Spec(R/gR)\right]_{d-1}$ follows.
		This completes the proof of the last part of the theorem.
	\end{proof}

	\begin{remark}
		\label{rem_formulas}
		Due to \autoref{thm_polar_segre_cycles}, we obtain the following consequences:
		\begin{enumerate}[\rm (i)]
			\item We have the equations
			\begin{enumerate}[\rm (a)]
				\item $\big[P_i(I, X)\big]_{d-i} \;=\; \big[P_1\left(I, P_{i-1}(I, X)\right)\big]_{d-i}$.
				\item $\Lambda_i(I, X) \;=\; \big[\Lambda_1\left(I, P_{i-1}(I, X)\right)\big]_{d-i}$.
				\item $V_i(I, X) \;=\; \big[V_1\left(I, P_{i-1}(I, X)\right)\big]_{d-i}$.
			\end{enumerate}
			\item We recover the formula of \cite[Theorem 3.4]{Trung2001} for the mixed multiplicities of an ideal
			$$
			m_i(I, R) \;=\; e_{d-i}\left(R/(g_1,\ldots,g_i):_RI^\infty\right).
			$$
			\item The new invariants (i.e., polar-Segre multiplicities) have the following formula
			$$
			\nu_i(I, R) \;=\; e_{d-i}\left(R/(g_1,\ldots,g_{i-1}):_RI^\infty+g_iR\right).
			$$
			\item We obtain the following formula for Segre numbers
			$$
			c_i(I, R) \;=\; \sum_{\substack{\pp \in V\left((g_1,\ldots,g_{i-1}):_RI^\infty\right)\\ \pp \in V(I), \; \dim(R/\pp)=d-i}} e\big(I,\, R_\pp/(g_1,\ldots,g_{i-1})R_\pp:_{R_\pp}I^\infty R_\pp\big)\cdot e(R/\pp).
			$$ 
			On the other hand, the formulas of part (i) and (iii) yield
			\begin{align*}
				c_i(I, R) &\;=\; e_{d-i}\left(R/(g_1,\ldots,g_{i-1}):_RI^\infty+g_iR\right) \;-\; e_{d-i}\left(R/(g_1,\ldots,g_i):_RI^\infty\right)\\
				&\;=\; e_{d-i}\left(\HH_I^0\big(R/(g_1,\ldots,g_{i-1}):_RI^\infty+g_iR\big)\right)\\
				&\;=\; \sum_{\substack{\pp \in V\left((g_1,\ldots,g_{i-1}):_RI^\infty\right)\\ \pp \in V(I), \; \dim(R/\pp)=d-i}} \length_{R_\pp}\left(
				\frac{R_\pp}
				{(g_1,\ldots,g_{i-1})R_\pp :_{R_\pp}  I^\infty R_\pp + g_iR_\pp} \right)\cdot e(R/\pp).
			\end{align*}
			This recovers the length formula of \cite[Proposition 2.1]{PTUV} for the Segre numbers of an ideal.
			Therefore, for any prime $\pp$ in the summation above, it follows that  
			$$
			e(I,\, R_\pp/(g_1,\ldots,g_{i-1})R_\pp:_{R_\pp}I^\infty R_\pp) \;=\; \length_{R_\pp}\left( R_\pp/\left((g_1,\ldots,g_{i-1})R_\pp:_{R_\pp} I^\infty R_\pp + g_iR_\pp\right)\right).
			$$
			In any case, this is expected: if the residue field $\kappa$ is not an algebraic extension of a finite field, we may assume that $g_i \in I$ is general in $IR_\pp$ for each prime $\pp$ above (see \cite[Lemma 2.6]{PTUV}).
		\end{enumerate}
	\end{remark}

	\section{Integral dependence and specialization}
	\label{sect_integral_dep}
	
	In this section, we discuss how integral dependence can be detected by utilizing the invariants that we study in \autoref{sect_invariants_cycles}.
	We also discuss when integral closure and polar schemes specialize.
	Moreover, we generalize the result of Gaffney and Gassler \cite{GG} regarding the lexicographic upper semicontinuity of Segre numbers (see \autoref{thm_Segre_lex_upper}).
	Throughout this section, we continue using \autoref{setup_polar_Segre}.
	
	To deal with the case where $\HT(I) = 0$, we also need to consider an additional number that is denoted as $c_0(I, R)$.
	Let $\mathscr{H} := \gr_\mm(G) = \gr_\mm\left(\gr_I(R)\right) $ with standard bigrading $[\mathscr{H}]_{(v, n)} = \mm^{n}G_v/\mm^{n+1}G_v$.
	The first sum transform of the Hilbert function of $\mathscr{H}$ with respect to $n$ is equal to 
	$$
	H_\mathscr{H}^1(v, n) := \sum_{k = 0}^{n}\dim_\kappa\left([\mathscr{H}]_{(v, n)}\right) = \length_R\left(G_v/\mm^{n+1}G_v\right),
	$$ 
	and it encodes the polar multiplicities of the standard graded algebra $G=\gr_I(R)$ (i.e., the Segre numbers $c_1(I, R), \ldots, c_d(I, R)$; see \autoref{sect_invariants_cycles}).
	If we further consider the second sum transform $H_\mathscr{H}^2(v, n) := \sum_{k = 0}^{v}H_\mathscr{H}^1(k,n)$, for $v \gg 0$ and $n \gg 0$, we get a polynomial 
	$$
	H_\mathscr{H}^2(v, n) \;=\; \sum_{i=0}^d\; \frac{c_i(I, R)}{i!\,(d-i!)}\; v^in^{d-i} \;+\; \text{(lower degree terms)},
	$$
	that encodes the Segre numbers $c_1(I, R),\ldots,c_d(I, R)$ defined in \autoref{sect_invariants_cycles}, but also the new number $c_0(I, R)$.
	This leads to the following definition.

	\begin{definition}[Achilles--Manaresi \cite{AM_MULT_SEQ}]
		The \emph{multiplicity sequence} of the ideal $I \subset R$ is given by 
		$$
		\big(c_0(I, R), c_1(I, R), \ldots, c_d(I, R)\big) \;\in\; \ZZ_{\ge 0}^{d+1}.
		$$
	\end{definition}
	
	In \autoref{sect_invariants_cycles}, we restricted ourselves to considering only the first sum transform $H_\mathscr{H}^1$ (i.e., the polar multiplicities of $G$) because then we could use many desirable properties of polar multiplicities  (see \cite{cidruiz2024polar}, \cite[\S 8]{KLEIMAN_THORUP_GEOM}).

	We now discuss when the polar schemes specialize module an element $t \in \mm$ that is part of a system of parameters of $R$ (i.e., $\dim(R/tR) =d-1$).
	We say that the $i$-th polar scheme $P_i(I, X)$ \emph{specializes} modulo the element $t \in \mm$ if $\dim\left(P_i(I, X) \cap V(t)\right) \le d-i-1$ and we have the equality of cycles
	$$
	\big[P_i(I, X) \cap V(t)\big]_{d-i-1} \;=\; \big[P_i\left(I, X \cap V(t)\right)\big]_{d-i-1} \;\in\;  Z_{d-i-1}(X).
	$$
	We shall see that the question of whether polar schemes specialize is governed by the Segre numbers $c_i(I, R)$ and $c_i(I, R/tR)$.
	First, we have the following basic but useful observation. 
	
	\begin{lemma}
		\label{lem_sp_polar}
		Assume \autoref{setup_segre_cycles}.
		Let $t \in \mm$ be such that $\dim(R/tR) = d-1$.
		If\, $\HT\left(I, (g_1,\ldots,g_i):_RI^\infty, t\right) \ge i +2$, then $P_i(I, X)$ specializes modulo $t$.
	\end{lemma}
	\begin{proof}
		Let $\overline{R} := R/\left((g_1,\ldots,g_i):_RI^\infty, t\right)$ and notice that $I^k\overline{R} \cdot \left((t, g_1,\ldots,g_i):_RI^\infty\right)\overline{R} =0$ for some $k > 0$.
		Thus the result follows as a direct consequence of \autoref{thm_polar_segre_cycles}(i).
	\end{proof}

	A very important technical result is the following theorem.
	Our arguments below follow closely the ones in the proof of \cite[Theorem 3.3]{PTUV}.

	\begin{theorem}
		\label{thm_specialization_polar}
		Assume \autoref{setup_segre_cycles} and that $R$ is equidimensional and catenary.
		Let $1 \le i \le d-1$.
		Let $t \in \mm$ be such that $\HT\left((g_1,\ldots,g_{i-1}):_RI^\infty, t\right) \ge i$ and $\HT\left(I, (g_1,\ldots,g_{i-1}):_RI^\infty, t\right) \ge i +1$.
		Then the following statements hold: 
		\begin{enumerate}[\rm (i)]
			\item $\HT\left((g_1,\ldots,g_{i}):_RI^\infty, t\right) \ge i+1$.
			\item $c_i(I, R/tR) \ge c_i(I, R)$.
			\item If $c_i(I, R/tR) = c_i(I, R)$, then $\HT\left(I, (g_1,\ldots,g_i):_RI^\infty, t\right) \ge i +2$.
		\end{enumerate}
	\end{theorem}
	\begin{proof}
		We may assume that $\dim(R/(g_1,\ldots,g_{i-1}):_RI^\infty)=d-i+1$ because $R$ is equidimensional and catenary (see \autoref{rem_filter_reg_seq}).
		By utilizing \autoref{thm_polar_segre_cycles}(ii), we may assume $i=1$ and substitute $R$ by $R/(g_1,\ldots,g_{i-1}):_RI^\infty$.
		So, we are free to assume the conditions $i = 1$, $\dim(R/tR) = d-1$, $0:_RI^\infty=0$ and $\HT(I, t) \ge 2$.
		Hence $g_1$ is a nonzerodivisor since it is a general element of $I$.
		Set $g := g_1$.
		Since $g$ is a general element of $I$, it follows that $\HT((g):_RI^\infty, t) \ge \HT(g, t) \ge 2$; thus settling part (i).
		We need to verify $c_1(I, R/tR) \ge c_1(I, R)$ to prove part (ii).
		Part (iii) follows by showing that if $c_1(I, R/tR)=c_1(I, R)$ then $\HT(I, (g):_RI^\infty, t) \ge 3$.
		
		The rest of the proof follows verbatim the arguments in \cite[page 962, proof of Theorem 3.3]{PTUV}.
		Consider the finite sets of primes $\Lambda := \lbrace \pp \in V(I, t) \mid \HT(\pp) = 2 \rbrace$ and $\Gamma := \lbrace \qqq \in V(I) \mid \HT(\qqq) = 1 \rbrace$.
		Thus, to prove the inequality $\HT(I, (g):_RI^\infty, t) \ge 3$, it suffices to show that for each $\pp \in \Lambda$ we have $\pp \not\supset (g):_RI^\infty$ or, equivalently, $IR_\pp \subset \sqrt{gR_\pp}$.
		For each $\pp \in \Lambda$, we define the finite sets
		$$
		\Sigma_\pp := \{\qqq \in {\rm Min}(g) \mid \qqq \subset \pp \} \quad \text{ and } \quad \Gamma_\pp := \lbrace \qqq \in V(I) \mid \qqq \subset \pp \text{ and } \HT(\qqq)=1 \rbrace.
		$$
		Notice that $\Gamma_\pp \subseteq \Sigma_\pp$ and that an equality holds if and only if $IR_\pp \subset \sqrt{gR_\pp}$.
		By utilizing the faithfully flat extension $R \rightarrow R[y]_{\mm R[y]}$, we may assume that the residue field of $R$ is not an algebraic extension of a finite field, and thus we may assume that $g$ is a general element of $IR_\pp$ for each $\pp \in \Lambda$ (see \cite[Lemma 2.6]{PTUV}).
		
		We have the following
		{\footnotesize
			\begin{alignat*}{3}
				c_1(I, R/tR) &= \sum_{\pp \in \Lambda} e\left(I, R_\pp/tR_\pp\right)\cdot e(R/\pp) &\text{by \cite[Proposition 2.4(a)]{PTUV} or \autoref{thm_polar_segre_cycles}(ii)}\\
				&= \sum_{\pp \in \Lambda} e\left((g), R_\pp/tR_\pp\right)\cdot e(R/\pp) &\text{since $g$ is general in $IR_\pp$} \\
				&\ge \sum_{\pp \in \Lambda} e\left((t), R_\pp/gR_\pp\right)\cdot e(R/\pp) &\text{by \cite[Lemma 3.1]{PTUV} since $g$ is a nonzerodivisor}\\
				&= \sum_{\pp \in \Lambda} \sum_{\qqq \in \Sigma_\pp} \length_{R_\qqq}(R_\qqq/gR_\qqq)\cdot e\left((t), R_\pp/\qqq R_\pp\right)\cdot e(R/\pp) & \text{by the associativity formula}\\
				&\ge \sum_{\pp \in \Lambda} \sum_{\qqq \in \Gamma_\pp} \length_{R_\qqq}(R_\qqq/gR_\qqq)\cdot e\left((t), R_\pp/\qqq R_\pp\right)\cdot e(R/\pp) & \text{since $\Gamma_\pp \subseteq \Sigma_\pp$}\\
				&= \sum_{\qqq \in \Gamma} \length_{R_\qqq}(R_\qqq/gR_\qqq) \sum_{\pp \in \Lambda, \pp \supset \qqq}  e\left((t), R_\pp/\qqq R_\pp\right)\cdot e(R/\pp) & \text{by switching the summation} \\
				&\ge \sum_{\qqq \in \Gamma} \length_{R_\qqq}(R_\qqq/gR_\qqq) \cdot e(R/\qqq) & \text{by \cite[Lemma 3.2]{PTUV}} \\
				&\ge \sum_{\qqq \in \Gamma} e((g), R_\qqq) \cdot e(R/\qqq) &\text{by \cite[Corollary 1.2.12]{FLENNER_O_CARROLL_VOGEL}}\\
				&\ge \sum_{\qqq \in \Gamma} e(IR_\qqq, R_\qqq) \cdot e(R/\qqq) & \text{since $g \in I$}\\
				&=c_1(I, R) &\text{by \cite[Proposition 2.4(a)]{PTUV} or \autoref{thm_polar_segre_cycles}(ii).}
			\end{alignat*}
		}
		Therefore, we obtain the inequality $c_1(I, R/tR) \ge c_1(I, R)$ and that $\Gamma_\pp = \Sigma_\pp$ if an equality holds.
		This concludes the proof of the theorem.
	\end{proof}

	Important consequences of \autoref{thm_polar_segre_cycles} and \autoref{thm_specialization_polar} are the following lexicographical upper-semicontinuity results.
	
	\begin{corollary}
		\label{cor_lex_order}
		Assume that $R$ is equidimensional and catenary, and let $t \in \mm$ be such that $\dim(R/tR) = d-1$.
		Suppose that $\HT(I, 0:_RI^\infty, t) \ge 2$.
		Then we get the inequality 
		$$
		\big(c_1(I, R/tR), c_2(I, R/tR),\ldots, c_{d-1}(I, R/tR)\big) \;\;\ge_{\rm lex}\;\; \big(c_1(I, R), c_2(I, R),\ldots, c_{d-1}(I, R)\big)
		$$
		under the lexicographical order.
	\end{corollary}

	\begin{corollary}
		\label{cor_lex_upper_two_ideals}
		Assume that $R$ is equidimensional and universally catenary.
		Let $I \subset J$ be two $R$-ideals.
		Then we get the inequality 
		$$
		\big(c_0(I, R), c_1(I, R),\ldots, c_d(I, R)\big) \;\;\ge_{\rm lex}\;\; \big(c_0(J, R), c_1(J, R),\ldots, c_d(J, R)\big)
		$$
		under the lexicographical order.
	\end{corollary}
	\begin{proof}
		By replacing $R$ by $R[y]_{(\mm, y)}$ (where $y$ is a new variable) and the ideals $I$ and $J$ by $(I, y)$ and $(J, y)$, we may assume that $\HT(I)\ge 1$ and $\HT(J) \ge 1$ and that $c_0(I, R) = 0$ and $c_0(J, R)=0$; see \cite[Corollary 2.2(d)]{PTUV}.
		
		Consider the local ring $S=R[t]_{(\mm, t)}$ (where $t$ is a new variable).
		Let $H = IS + tJS \subset S$.
		Notice that $\HT(H, 0:_SH^\infty, t) \ge \HT(IS, t) \ge 2$, $H\cdot S/tS = I$ and $HS_{\mm S} = JR[t]_{\mm R[t]}$.
		Then we obtain the following inequalities 
		\begin{alignat*}{3}
			\big(c_1(I, R), c_2(I, R),\ldots, c_d(I, R)\big) &\;=\; \big(c_1(H, S/tS), c_2(H, S/tS),\ldots, c_d(H, S/tS)\big) &  \\
			&\;\ge_{\rm lex}\; \big(c_1(H, S), c_2(H, S/tS),\ldots, c_d(H, S)\big) & \text{ by \autoref{cor_lex_order}} \\ 
			&\;\ge\; \big(c_1(H, S_{\mm S}), c_2(H, S_{\mm S}),\ldots, c_d(H, S_{\mm S})\big) & \text{ by \cite[Proposition 2.7]{PTUV}} \\
			&\;=\; \big(c_1(J, R), c_2(J, R),\ldots, c_d(J, R)\big) &
		\end{alignat*}
		that settle the claim of the corollary.
	\end{proof}
	
	We now present a principle of specialization of integral dependence that we enunciate similarly to the ones in \cite[Theorem 4.7]{GG} and \cite[Theorem 4.4]{PTUV}.
	
	\begin{theorem}[Principle of Specialization of Integral Dependence -- PSID]
		\label{thm_psid}
		Assume \autoref{setup_polar_Segre} and that $R$ is equidimensional and universally catenary. 
		Let $t \in \mm$ be such that $\dim(R/tR) = d-1$.
		Suppose that $\HT(I, t) \ge 2$ and that $c_i(I, R/tR) = c_i(I, R)$ for all $1\le i \le d-1$.
		Then, for any element $h \in R$, the following two conditions are equivalent:
		\begin{enumerate}[\rm (i)]
			\item $h \in \overline{I}$.
			\item $hR_\pp \in \overline{IR_\pp}$ for all primes $\pp \in \Spec(R)$ such that $t \not\in \pp$.
		\end{enumerate}
	\end{theorem}
	\begin{proof}
		Assume momentarily that $\dim(G/tG)<d$.
		As $G$ is equidimensional of dimension $d$ (see \cite[Theorem 3.8]{RATLIFF_QUNMIXED_II}, \cite[Lemma 2.2]{SUV_MULT}), it follows that the contraction of a minimal prime of $G$ to $R$ does not contain $t$.
		This implies that no prime in the set $L(I):=\lbrace\pp \in V(I) \mid \dim(R_\pp) = \ell(IR_\pp) \rbrace$ contains $t$.
		Since every associated prime of $\overline{I}$ belongs to $L(I)$ (see \cite[3.9 and 4.1]{MCADAM}), the result of the theorem follows.
		Hence it remains to show $\dim(G/tG)<d$.
		
		Since $\HT(I) \ge 1$, it follows that $\dim(G/G_+) = \dim(R/I) \le d-1$ and so no minimal prime of $G$ contains $G_+$.
		Set $E' := \Proj(G/tG)$.
		Thus $\dim(E') < d-1$ if and only if $\dim(G/tG) < d$.
		From \cite[Proposition 2.10]{cidruiz2024polar}, we obtain $m_{d-1}^{d-1}(G/tG) = e_{d-1}(\HH^0(E', \OO_{E'}))$ and so it follows that $m_{d-1}^{d-1}(G/tG) = 0$ because $\Supp(\HH^0(E', \OO_{E'})) \subset V(I, t)$ and by assumption we have the inequality $\HT(I, t) \ge 2$.
		We now assume the notation and assumptions of \autoref{setup_segre_cycles}.
		From \cite[Proposition 2.10]{cidruiz2024polar} and \autoref{eq_specialization_E}, we get 
		$$
		m_{d-1}^{d-1-i}(G/tG) \;=\; m_{d-1-i}^{d-1-i}\left(G/(\ell_1,\ldots,\ell_i,t)G\right) \;=\; m_{d-1-i}^{d-1-i}(G_i/tG_i),
		$$
		where $G_i = \gr_{IR_{i}}(R_i)$ and $R_i = R / (g_1,\ldots,g_i)R$.
		Therefore, by utilizing \autoref{thm_polar_segre_cycles}(ii) and \autoref{thm_specialization_polar}, we can prove inductively that 
		$$
		m_{d-1}^{d-1-i}(G/tG) = 0 \quad \text{ for all } \quad 0 \le i \le d-1.
		$$
		Due to \cite[Lemma 2.3]{cidruiz2024polar}, we obtain $\dim(E') < d-1$, and so it follows that $\dim(G/tG) < d$.		
	\end{proof}

	The following theorem contains our new criteria for integral dependence in terms of the invariants $m_i(I, R)$ and $\nu_i(I, R)$. 
	
	\begin{theorem}
		\label{thm_new_crit_int_dep}
		Assume \autoref{setup_polar_Segre} and that $R$ is equidimensional and universally catenary. 
		Let $I \subset J$ be two $R$-ideals.
		Suppose the following two conditions: 
		\begin{enumerate}[\rm (a)]
			\item $o(I) = o(J)$.
			\item $I$ satisfies the $\sG$-parameter condition generically {\rm(}see \autoref{nota_G_param}{\rm)}.
		\end{enumerate}
		Then the following are equivalent: 
		\begin{enumerate}[\rm (i)]
			\item $J$ is integral over $I$.
			\item $m_i(I, R) = m_i(J,R)$ for all $0 \le i \le d-1$.
			\item $\nu_i(I, R) = \nu_i(J, R)$ for all $0 \le i \le d$.
		\end{enumerate}
	\end{theorem}
	\begin{proof}
		Let $\delta := o(I) = o(J)$ be the order of $I$ and $J$.
		
		(i) $\Rightarrow$ (ii) \& (iii):
		If $J$ is integral over $I$, then it is known that $m_i(I, R) = m_i(J, R)$ and $c_i(I, R) = c_i(J, R)$ (see \cite[Corollary 3.8]{Trung2001} and \cite[Theorem 4.2]{PTUV}). 
		This settles both implications (without assuming the conditions (a) and (b)). 
		
		(ii) $\Rightarrow$ (i): Suppose that $m_i(I, R) = m_i(J,R)$ for all $0 \le i \le d-1$.
		Then \autoref{prop_decomp_div} yields
		$$
		\nu_i(I, R) \;=\; \delta \cdot m_{i-1}(I, R) \;=\; \delta \cdot m_{i-1}(J, R) \;\le\; \nu_i(J, R) 
		$$
		for all $1 \le i \le d$.
		This implies that $c_i(I, R) = \nu_i(I, R) - m_i(I, R) \le \nu_i(J, R) - m_i(J, R) = c_i(J, R)$ for all $1 \le i \le d$.
		On the other hand, we have $c_0(I, R) = e(R) - m_0(I, R) = e(R) - m_0(J, R) = c_0(J, R)$.
		Therefore, $J$ is integral over $I$ due to \cite[Theorem 4.2]{PTUV}.
		
		(iii) $\Rightarrow$ (i):  Suppose that $\nu_i(I, R) = \nu_i(J,R)$ for all $0 \le i \le d$.
		From \autoref{prop_decomp_div}, we obtain
		$$
		\delta \cdot m_{i-1}(I, R) \;=\; \nu_i(I, R) \;=\; \nu_i(J, R) \;\ge\;   \delta \cdot m_{i-1}(J, R)
		$$
		for all $1 \le i \le d$.
		It follows that $c_i(I, R) = \nu_i(I, R) - m_i(I, R) \le \nu_i(J, R) - m_i(J, R) = c_i(J, R)$ for all $0 \le i \le d-1$.
		We also have $c_d(I, R) = \nu_d(I, R) = \nu_d(J, R) = c_d(J, R)$.
		Finally, $J$ is integral over $I$ by invoking \cite[Theorem 4.2]{PTUV} once again.
	\end{proof}
	
	The next remark complements \cite{Trung2001}*{Corollary 4.3}.
	\begin{remark}
		Let $R$ be a regular local ring of dimension at least two and let $I$ be an ideal of height one. Write $I=aH$ where $H$ is an ideal with $\HT (H)\ge 2$.
		Then $m_1(I, R)=o(H)$. 
	\end{remark}
	\begin{proof} We may assume that the residue field of $R$ is infinite. Let $g$ be a general element of $H$. According to \autoref{thm_polar_segre_cycles}(i), we have $m_1(I,R)=e(R/agR:_RI^{\infty}).$ 
		On the other hand, 
		$$
		agR:I^{\infty} \;=\; (agR:_RaH):_R (aH)^{\infty} \;=\; (gR:H):_R(aH)^{\infty} \;=\; (gR:_RH^{\infty}):_R(aR)^{\infty} \;=\; gR.
		$$
		Finally, $e(R/gR)=o(g)=o(H)$.
	\end{proof}
	
	The remark above shows that the polar multiplicities of  an ideal of height one in a two-dimensional regular local ring only depend on the order of the saturation of the ideal with respect to its unmixed part, and hence cannot detect integral dependence. 
	
	\begin{example}[Polar multiplicities do not detect integral dependence]
		\label{counter_polar}
		Let $R = \kk[x_0,x_1]$ be a polynomial ring over a field.
		Consider the ideals $I = (x_{0}^{2},\,x_{0}x_{1}^2) \,\subsetneq\, J = (x_{0}^{2},\,x_{0}x_{1})$. 
		We have the following table of values:
		\begin{equation*}
			\arraycolsep=8.2pt
			\begin{array}{l|r}
				\begin{array}{lll}
					c_0(I, R) = 0 & m_0(I, R) = 1 & \nu_0(I, R) = 1\\
					c_1(I, R) = 1 & m_1(I, R) = 1 & \nu_1(I, R) = 2 \\
					c_2(I, R) = 4 & m_2(I, R) =0 & \nu_2(I, R) = 4 \\
				\end{array}
				&
				\begin{array}{lll}
					c_0(J, R) = 0 & m_0(J, R) = 1 & \nu_0(J, R) = 1\\
					c_1(J, R) = 1 & m_1(J, R) = 1 & \nu_1(I, R) = 2 \\
					c_2(J, R) = 2 & m_2(J, R) =0 & \nu_2(I, R) = 2. \\
				\end{array}
			\end{array}			
		\end{equation*}
		Here the invariants $m_i$ fail to detect the fact that $\overline{I} \subsetneq \overline{J}$.
		However, both invariants $c_i$ and $\nu_i$ do detect this.
	\end{example}

	From the previous example, we see that the invariants $\nu_i$ seem to carry more information related to the integral closure of an ideal. 
	Nevertheless, we have the following more delicate example where the invariants $\nu_i$ fail to detect integral dependence. 
	
	\begin{example}[Polar-Segre multiplicities do not detect integral dependence]
		\label{counter_nu}
		Let $\PP_{\kk}^1 = \Proj(\kk[s,t]) \rightarrow \PP_{\kk}^3=\Proj(\kk[x_0,\ldots,x_3])$, $(s:t) \mapsto (s^4: s^3t :st^3:t^4)$ be the parametrization of the rational quartic $\mathcal{C} \subset \PP_{\kk}^3$.
		Let $R = \kk[x_0,\ldots,x_3] / \fP$ be the homogeneous coordinate ring of $\mathcal{C}$.
		Consider the ideals $I = (x_{0}^{2},\,x_{0}x_{2}x_{3},\,x_{1}^{2}x_{3}^{2})R \,\subsetneq\, J = (x_{0}^{2},\,x_{0}x_{2},\,x_{0}x_{3},\,x_{1}x_{3}^{2})R$.
		We get the following table of values:
		\begin{equation*}
			\arraycolsep=8.2pt
			\begin{array}{l|r}
				\begin{array}{lll}
					c_0(I, R) = 0 & m_0(I, R) = 4 & \nu_0(I, R) = 4\\
					c_1(I, R) = 5 & m_1(I, R) = 3 & \nu_1(I, R) = 8 \\
					c_2(I, R) = 14 & m_2(I, R) =0 & \nu_2(I, R) = 14 \\
				\end{array}
				&
				\begin{array}{lll}
					c_0(J, R) = 0 & m_0(J, R) = 4 & \nu_0(J, R) = 4\\
					c_1(J, R) = 3 & m_1(J, R) = 5 & \nu_1(I, R) = 8 \\
					c_2(J, R) = 14 & m_2(J, R) =0 & \nu_2(I, R) = 14. \\
				\end{array}
			\end{array}			
		\end{equation*}
		Therefore the invariants $\nu_i$ fail to detect the fact that $\overline{I} \subsetneq \overline{J}$.
	\end{example}

	\subsection{Lexicographic upper semicontinuity of Segre numbers}\hfill\\
	\label{subsect_lex}
	
	In this subsection, we show that Segre numbers have a lexicographic upper semicontinuous behavior in families. 
	This generalizes a result due to Gaffney and Gassler \cite{GG}.
	We now use the following setup. 
	
	\begin{setup}
		\label{setup_upper_semicont}
		Let $\iota : A \hookrightarrow R$ be a flat injective homomorphism of finite type of Noetherian rings.
		Suppose that the inclusion $\iota$ has a section; i.e., there is a homomorphism $\pi : R \surjects A$ such that $\pi \circ \iota = {\rm id}_A$.
		Consider the ideal $Q := \Ker(\pi) \subset R$.
	\end{setup}

	Notice that given prime $\pp \in \Spec(A)$, we always have that $QR(\pp)$ is a prime ideal in the fiber $R(\pp)$.
	Indeed, we get the inclusion $\iota_\pp : A/\pp \hookrightarrow R/\pp R$ with section $\pi_\pp : R/\pp R \surjects A/\pp$, and since $\Ker(\pi_\pp) = {Q\cdot R/\pp R}$, it follows that ${Q\cdot R/\pp R}$ is a prime ideal satisfying $\left(R/\pp R\right)_{Q\cdot R/\pp R} = R(\pp)_{QR(\pp)}$.
	Moreover ${QR(\pp)}$ is a maximal ideal in $R(\pp)$.
	Thus, although the fiber $R(\pp)$ might not be a local ring,  we do get a distinguished local ring $R(\pp)_{QR(\pp)}$.
	To simplify notation, we say that $S(\pp) := R(\pp)_{QR(\pp)}$ is the \emph{distinguished fiber} of $\pp$.
	
	Our main result regrading the behavior of Segre numbers in families is the following.	
	
	\begin{theorem}
		\label{thm_Segre_lex_upper}
		Assume \autoref{setup_upper_semicont} and let $I \subset R$ be an ideal. 
		Assume that for all $\pp \in \Spec(A)$, the fibers $R(\pp)$ are equidimensional of the same dimension $d$ and
		$\HT(I(\pp)) > 0$.
		Then the function 
		$$
		\pp \in \Spec(A) \;\mapsto\; \big(c_1\left(I, S(\pp)\right), c_2\left(I, S(\pp)\right), \ldots, c_d\left(I, S(\pp)\right)\big) \in \ZZ_{\ge 0}^d 
		$$
		is upper semicontinuous with respect to the lexicographic order.
	\end{theorem}
	\begin{proof}
		By the topological Nagata criterion (see \autoref{rem_top_Nagata}), it suffices to show the following two conditions: 
		\begin{enumerate}[\rm (i)]
			\item Under the assumption that $A$ is a domain, there is a dense open subset $U \subset \Spec(A)$ where the above function is constant. 			
			\item If $\pp \supset \qqq$, then $\big(c_1(I, S(\pp)),\ldots,S(\pp)\big) \ge_{\rm lex} \big(c_1(I, S(\qqq)),\ldots,c_d(I, S(\qqq))\big)$.
		\end{enumerate}
		
		First, we prove the claim (i).
		We assume $A$ is a domain, thus by Grothendieck's Generic Freeness Lemma (see, e.g., \cite[Theorem 24.1]{MATSUMURA}, \cite[Theorem 14.4]{EISEN_COMM}) there is an element $0 \neq a \in A$ such that the bigraded components of $\gr_{Q}(\gr_I(R)) \otimes_A A_a$ and the graded components of $\gr_I(R) \otimes_A A_a$ are free $A_a$-modules.
		Set $U := D(a) \subset \Spec(A)$.
		Then, for any $\pp \in U$, we have the isomorphisms 
		$$
		\gr_{Q}(\gr_I(R)) \otimes_A \kappa(\pp) \;\cong\; \gr_{Q}(\gr_I(R) \otimes_A \kappa(\pp)) \;\cong\; \gr_{Q}(\gr_I(R(\pp))).
		$$
		This shows shows that the above function is constant on $U$.
		
		Next, we prove the claim (ii).
		Notice that we can reduce modulo $\qqq$ and localize at $\pp$.
		Thus we assume $A$ is a local domain with maximal ideal $\pp$ and $\qqq = 0$.
		Let $K := \Quot(A)$.
		By \cite[Exercise II.4.11]{HARTSHORNE} or \cite[Proposition 7.1.7]{EGAII}, there is a discrete valuation ring $V$ of $K$ that dominates $A$; that is, $A \subset V$ and $\pp  = \mathfrak{n} \cap A$ where $\mathfrak{n} = (t) \subset V$ is the closed point of $\Spec(V)$. 
		Let $R_V := R \otimes_A V$.
		Notice that  $\iota' : V \hookrightarrow R_V$ is flat and has a section $\pi' : R_V \surjects V$ with $\Ker(\pi') = QR_V$.
		We get a field extension $\kappa(\pp) \hookrightarrow \kappa(\nnn)$ and the following isomorphisms
		$$
		R_V(\nnn) \;\cong\; (R \otimes_A V) \otimes_V \kappa(\nnn) \;\cong\; R \otimes_A \kappa(\nnn) \;\cong\; (R \otimes_A \kappa(\pp)) \otimes_{\kappa(\pp)} \kappa(\nnn) \;\cong\; R(\pp) \otimes_{\kappa(\pp)} \kappa(\nnn).
		$$
		This gives the isomorphisms $I^vR_V(\nnn) \cong  I^vR(\pp) \otimes_{\kappa(\pp)} \kappa(\nnn)$ and so it follows		
		\begin{equation}
			\label{eq_gr_base_change}
			\gr_Q\left(\gr_I\left(R_V(\nnn)\right)\right) \;\cong\; \gr_Q\left(\gr_I\left(R(\pp)\right)\right) \otimes_{\kappa(\pp)} \kappa(\nnn).
		\end{equation}

		Let $T:={(R_V)}_{(t, Q)R_V}$.
		From \autoref{eq_gr_base_change}, we have 
		\begin{equation}
			\label{eq_first_c_i}
			c_i(I, S(\pp)) \;=\; c_i(I, T/tT).
		\end{equation}
		Since $T_{QT} = R(\qqq)_{QR(\qqq)} = S(\qqq)$, we also obtain 
		\begin{equation}
			\label{eq_second_c_i}
			c_i(I, S(\qqq)) \;=\; c_i(I, T_{QT}).
		\end{equation}
		
		To use \autoref{cor_lex_order}, we now verify that the ideal $IT \subset T$ satisfies the required assumptions. 
		By applying \autoref{lem_properties_T}(i) to the ring $R(\pp)$, we conclude that the fibers of $R_V$ over $V$ are equidimensional of dimension $d$.
		Hence due to \autoref{lem_properties_T}(ii), $T$ is equidimensional of dimension $d+1$ and universally catenary and $\dim(T/tT) = \dim(T)-1$.
		Since $R(\pp) \hookrightarrow T/tT$ is a flat  homomorphism, $\HT((I,t)T/tT) \ge \HT(I(\pp)) \ge 1$, hence $\HT(t, IT) \ge 2$ because $t$ is regular on $T$.
		Now  \autoref{eq_first_c_i} and \autoref{cor_lex_order} yield
		$$
		\big(c_1(I, S(\pp)), \ldots, c_d(I, S(\pp))\big) \;=\; \big(c_1(I, T/tT), \ldots, c_d(I, T/tT)\big) \;\ge_{\rm lex}\; \big(c_1(I, T), \ldots, c_d(I, T)\big).
		$$
		
		Finally, from \cite[Proposition 2.7]{PTUV} and \autoref{eq_second_c_i}, we get 
		$$
		\big(c_1(I, T), \ldots, c_d(I, T)\big) \;\ge\; \big(c_1(I, T_{QT}), \ldots, c_d(I, T_{QT})\big) \;=\; \big(c_1(I, S(\qqq)), \ldots, c_d(I, S(\qqq))\big).
		$$
		This concludes the proof of the theorem.
	\end{proof}

	\begin{remark}
		A basic situation where the assumptions of \autoref{setup_upper_semicont} hold is the following.
		Assume $R=A[x_1,\ldots,x_d]$ is a positively graded polynomial ring over $A$, and let $\mm = (x_1,\ldots,x_d)$ be the graded irrelevant ideal.
		In this case, the natural projection $\pi : R \surjects A \cong R/\mm$ is a section of the inclusion $\iota : A \rightarrow R$. 
		Therefore, given a prime $\pp \in \Spec(A)$, the distinguished fiber of $\pp$ is the localization $R(\pp)_{\mm R(\pp)}$ of the fiber $R(\pp)\cong\kappa(\pp)[x_1,\ldots,x_d]$ at the graded irrelevant ideal $\mm R(\pp)$.
	\end{remark}

	The following technical lemma was used in the proof of \autoref{thm_Segre_lex_upper}.	
	
	\begin{lemma}
		\label{lem_properties_T}
		The following statements hold: 
		\begin{enumerate}[\rm (i)]
			\item Let $B$ be a finitely generated algebra over a field $\kk$.
			If $B$ is equidimensional, then $B \otimes_\kk L$ is equidimensional of the same dimension for any field extension $L$ of\, $\kk$.
			\item Let $V$ be a discrete valuation ring with uniformizing parameter $t$.
			Let $j : V \hookrightarrow B$ be a flat injective homomorphism of finite type that has a section $\tau : B \surjects V$, and let $\mathfrak{Q} = \Ker(\tau)$.
			Let $T = B_{(t, \mathfrak{Q})}$.
			Suppose that the fibers of $j$ are equidimensional of dimension $e$. 		
			Then $T$ is equidimensional of dimension $e+1$ and universally catenary, and $\dim(T/tT) = e$.
		\end{enumerate}
	\end{lemma}
	\begin{proof}
		(i) By \cite[Theorem 23.2]{MATSUMURA}, we have $\Ass\left(B\otimes_\kk L\right) = \bigcup_{\pp \in \Ass(B)} \Ass\left(B/\pp \otimes_\kk L\right)$.
		The result follows because each minimal prime of $B/\pp \otimes_\kk \LL$ has the same dimension as $B/\pp$ (see \cite[Exercise II.3.20(f)]{HARTSHORNE}).
		
		(ii)  
		Let $K = \Quot(V)$.
		Since $B$ is $V$-flat, it follows that $t$ is a nonzerodivisor on $B$.
		Let $\fP \in \Spec(B)$ be a minimal prime contained in $(t,\mathfrak{Q})$.
		Such $\fP$ does not contain $t$, hence it corresponds to a minimal in the fiber $B \otimes_V K$.
		By applying the dimension formula \cite[Theorems 15.5, 15.6]{MATSUMURA} to the inclusion $V \hookrightarrow B/\fP$, we obtain 
		$$
		\dim\left((B/\fP)_{(t, \mathfrak{Q})}\right) \;=\; \dim(V) + \trdeg_{V}(B/\fP) = 1 + e.
		$$	
		This implies that $T$ is equidimensional of dimension $e+1$.
		As $T$ is finitely generated over a discrete valuation ring, it is universally catenary.
		Observe that $\dim(T/tT) = e$ since $t$ is regular on $T$.
	\end{proof}

	\subsection{The case of equigenerated ideals}\hfill\\
	\label{subsect_equigen}
	
	This brief subsection is dedicated to the family of equigenerated ideals, where our results are particularly satisfying.

	\begin{corollary}
		\label{cor_equigen}
		Let $\kk$ be a field, $R$ be a finitely generated standard graded $\kk$-algebra of dimension $d$ and $\mm = R_+$ be the graded irrelevant ideal.
		Let $I \subset R$ be a homogeneous ideal generated by elements of the same degree $\delta\ge 1$.
		Then the following statements hold: 
		\begin{enumerate}[\rm (1)]
			\item $\delta \cdot  m_{i-1}(I, R) = \nu_i(I, R)$ for all $1 \le i \le d$. 
			\item For any homogeneous ideal $J \subset R$ containing $I$ such that $o(J)=\delta$, the following are equivalent: 
			\begin{enumerate}[\rm (i)]
				\item $J$ is integral over $I$.
				\item $m_i(I, R) = m_i(J,R)$ for all $0 \le i \le d-1$.
				\item $\nu_i(I, R) = \nu_i(J, R)$ for all $0 \le i \le d$.
			\end{enumerate}
		\end{enumerate} 
	\end{corollary}	
	\begin{proof}[First proof.]
		We may assume that $\kk$ is infinite. 
		Notice that part (1) implies part (2) (indeed, by \autoref{prop_decomp_div}(iii), the statement (1) is equivalent to the condition (b) of \autoref{thm_new_crit_int_dep}).
		Since $I$ is equigenerated in degree $\delta$, we see the Rees algebra $\Rees(I)$ as a standard bigraded algebra with $\big[\Rees(I)\big]_{(v, n)} = \big[I^v\big]_{n+v\delta}$.
		Thus Nakayama's lemma yields the isomorphisms 
		$$
		\big[\gr_\mm\left(\Rees(I)\right)\big]_{(v, n)} \;=\; \frac{\mm^nI^v}{\mm^{n+1}I^v} \;\cong\; \big[I^v\big]_{n+v\delta} \;=\; \big[\Rees(I)\big]_{(v, n)}.
		$$
		Let $g$ be a general element of $I$.
		Since $\left[\left(0:_{\Rees(I)}g\right)\right]_{(v, *)}=0$ for $v \gg 0$ (see \autoref{rem_pullback_g}), it follows that $\left[\left(0:_{\gr_\mm(\Rees(I))}\iniTerm(g)\right)\right]_{(v, *)}=0$ for $v \gg 0$.
		The result follows from \autoref{prop_decomp_div}(iii).
	\end{proof}
	\begin{proof}[Second proof.]
		Again, we see $\Rees(I)$ as a standard bigraded algebra with $\big[\Rees(I)\big]_{(v, n)} = \big[I^v\big]_{n+v\delta}$.
		Therefore $\gr_I(I) \cong \Rees(I)/I\Rees(I)$ is also naturally a standard bigraded algebra.
		Notice that we have a short exact sequence 
		$$
		0 \rightarrow \Rees(I)(1,-\delta) \rightarrow \Rees(I) \rightarrow \gr_I(R) \rightarrow 0.
		$$
		Then a basic computation with the polynomial functions $P_{\Rees(I)}(v, n) = \sum_{k=0}^n\dim_{\kk}\left([\Rees(I)]_{(v,k)}\right)$ and $P_{\gr_I(R)}(v, n) = \sum_{k=0}^n\dim_{\kk}\left([\gr_I(R)]_{(v,k)}\right)$ gives the equality 
		$$
		\delta \cdot  m_{i-1}(I, R) \;=\; m_i(I, R) + c_i(I, R) \;=\; \nu_i(I, R)
		$$ 
		for all $1 \le i \le d$ (see \cite[Remark 2.9]{cidruiz2024polar}). 
	\end{proof}

	As a consequence, in the equigenerated case, we get an alternative proof of the lexicographic upper semicontinuity of Segre numbers (see \autoref{thm_Segre_lex_upper}).

	\begin{remark}
		Let $A$ be a Noetherian domain and $R = A[x_0,\ldots,x_r]$ be a standard graded polynomial ring. 
		Let $\FF :  \PP_A^r \dashrightarrow \PP_A^s$ be a rational map with representative $\mathbf{f} = (f_0:\cdots:f_s)$ such that $f_0,\ldots,f_s$ are homogeneous elements of degree $\delta > 0$.
		Let $I = (f_0,\ldots,f_s) \subset R$ and assume that $I(\pp) \neq 0$ for all $\pp \in \Spec(A)$.
		Notice that $d_i(\FF(\pp)) = m_{r-i}(I, R(\pp))$ for all $0 \le i \le r$ and $\pp \in \Spec(A)$.
		By \autoref{thm_specialization_rat_map}, each function $\pp \mapsto m_i(I, R(\pp))$ is lower semicontinuous. 
		Notice that we always have $m_0(I, R(\pp)) = 1$.
		Therefore, the equality $c_i(I, R(\pp)) = \delta\cdot m_{i-1}(I, R(\pp)) - m_i(I, R(\pp))$ implies that the function 
		$$
		\pp \in \Spec(A) \;\mapsto\; \big(c_1\left(I, R(\pp)\right), c_2\left(I, R(\pp)\right), \ldots, c_{r+1}\left(I, R(\pp)\right)\big) \in \ZZ_{\ge 0}^{r+1} 
		$$
		is upper semicontinuous with respect to the lexicographic order.
	\end{remark}

	\section*{Acknowledgments}
	
	The authors were partially supported by NSF grant DMS-1928930 and Alfred P. Sloan Foundation G-2021-16778 while in residence at SLMath.

	\bibliography{references}

\begin{bibdiv}
\begin{biblist}

\bib{ACHILLES_MANARESI_J_MULT}{article}{
      author={Achilles, R.},
      author={Manaresi, M.},
       title={Multiplicity for ideals of maximal analytic spread and
  intersection theory},
        date={1993},
     journal={J. Math. Kyoto Univ.},
      volume={33},
      number={4},
       pages={1029\ndash 1046},
}

\bib{AM_MULT_SEQ}{article}{
      author={Achilles, R.},
      author={Manaresi, M.},
       title={Multiplicities of a bigraded ring and intersection theory},
        date={1997},
        ISSN={0025-5831},
     journal={Math. Ann.},
      volume={309},
      number={4},
       pages={573\ndash 591},
         url={https://doi.org/10.1007/s002080050128},
      review={\MR{1483824}},
}

\bib{Bhattacharya}{article}{
      author={Bhattacharya, P.~B.},
       title={The {H}ilbert function of two ideals},
        date={1957},
     journal={Proc. Cambridge Philos. Soc.},
      volume={53},
       pages={568\ndash 575},
}

\bib{BOEGER}{article}{
      author={B\"{o}ger, Erwin},
       title={Einige {B}emerkungen zur {T}heorie der ganz-algebraischen
  {A}bh\"{a}ngigkeit von {I}dealen},
        date={1970},
        ISSN={0025-5831},
     journal={Math. Ann.},
      volume={185},
       pages={303\ndash 308},
         url={https://doi.org/10.1007/BF01349952},
      review={\MR{263809}},
}

\bib{LORENTZIAN}{article}{
      author={Br\"{a}nd\'{e}n, Petter},
      author={Huh, June},
       title={Lorentzian polynomials},
        date={2020},
        ISSN={0003-486X,1939-8980},
     journal={Ann. of Math. (2)},
      volume={192},
      number={3},
       pages={821\ndash 891},
         url={https://doi.org/10.4007/annals.2020.192.3.4},
      review={\MR{4172622}},
}

\bib{MULTPROJ}{article}{
      author={Bus\'{e}, Laurent},
      author={Cid-Ruiz, Yairon},
      author={D'Andrea, Carlos},
       title={Degree and birationality of multi-graded rational maps},
        date={2020},
        ISSN={0024-6115},
     journal={Proc. Lond. Math. Soc. (3)},
      volume={121},
      number={4},
       pages={743\ndash 787},
         url={https://doi.org/10.1112/plms.12336},
      review={\MR{4105786}},
}

\bib{CS_PAPER}{article}{
      author={Cartwright, Dustin},
      author={Sturmfels, Bernd},
       title={The {H}ilbert scheme of the diagonal in a product of projective
  spaces},
        date={2010},
        ISSN={1073-7928},
     journal={Int. Math. Res. Not. IMRN},
      number={9},
       pages={1741\ndash 1771},
         url={https://doi.org/10.1093/imrn/rnp201},
      review={\MR{2643580}},
}

\bib{POSITIVITY}{article}{
      author={Castillo, Federico},
      author={Cid-Ruiz, Yairon},
      author={Li, Binglin},
      author={Monta{\~n}o, Jonathan},
      author={Zhang, Naizhen},
       title={When are multidegrees positive?},
        date={2020},
     journal={Advances in Mathematics},
      volume={374},
       pages={107382},
}

\bib{SPECIALIZATION_MARC_ARON}{article}{
      author={Chardin, Marc},
      author={Cid-Ruiz, Yairon},
      author={Simis, Aron},
       title={Generic freeness of local cohomology and graded specialization},
        date={2021},
        ISSN={0002-9947},
     journal={Trans. Amer. Math. Soc.},
      volume={375},
      number={01},
       pages={87\ndash 109},
         url={https://doi.org/10.1090/tran/8316},
      review={\MR{4358663}},
}

\bib{MIXED_MULT}{article}{
      author={Cid-Ruiz, Yairon},
       title={Mixed multiplicities and projective degrees of rational maps},
        date={2021},
        ISSN={0021-8693},
     journal={J. Algebra},
      volume={566},
       pages={136\ndash 162},
         url={https://doi.org/10.1016/j.jalgebra.2020.08.037},
      review={\MR{4151547}},
}

\bib{cidruiz2024polar}{misc}{
      author={Cid-Ruiz, Yairon},
       title={Polar multiplicities and integral dependence},
        date={2024},
}

\bib{MULT_SAT_FIB_GOR_3}{article}{
      author={Cid-Ruiz, Yairon},
      author={Mukundan, Vivek},
       title={Multiplicity of the saturated special fiber ring of height three
  {G}orenstein ideals},
        date={2021},
     journal={Acta Math. Vietnam.},
      volume={46},
      number={4},
       pages={663\ndash 674},
}

\bib{SPECIALIZATION_ARON}{article}{
      author={Cid-Ruiz, Yairon},
      author={Simis, Aron},
       title={{Degree of Rational Maps and Specialization}},
        date={202008},
        ISSN={1073-7928},
     journal={International Mathematics Research Notices},
         url={https://doi.org/10.1093/imrn/rnaa183},
        note={rnaa183},
}

\bib{CDNG_MINORS}{article}{
      author={Conca, Aldo},
      author={De~Negri, Emanuela},
      author={Gorla, Elisa},
       title={Universal {G}r\"{o}bner bases for maximal minors},
        date={2015},
        ISSN={1073-7928},
     journal={Int. Math. Res. Not. IMRN},
      number={11},
       pages={3245\ndash 3262},
         url={https://doi.org/10.1093/imrn/rnu032},
      review={\MR{3373050}},
}

\bib{CDNG_CS_IDEALS}{article}{
      author={Conca, Aldo},
      author={De~Negri, Emanuela},
      author={Gorla, Elisa},
       title={Universal {G}r\"{o}bner bases and {C}artwright-{S}turmfels
  ideals},
        date={2020},
        ISSN={1073-7928},
     journal={Int. Math. Res. Not. IMRN},
      number={7},
       pages={1979\ndash 1991},
         url={https://doi.org/10.1093/imrn/rny075},
      review={\MR{4089441}},
}

\bib{COX_KPU}{article}{
      author={Cox, David},
      author={Kustin, Andrew~R.},
      author={Polini, Claudia},
      author={Ulrich, Bernd},
       title={A study of singularities on rational curves via syzygies},
        date={2013},
     journal={Mem. Amer. Math. Soc.},
      volume={222},
      number={1045},
       pages={x+116},
}

\bib{DOLGACHEV}{book}{
      author={Dolgachev, Igor~V.},
       title={Classical algebraic geometry},
   publisher={Cambridge University Press, Cambridge},
        date={2012},
        note={A modern view},
}

\bib{AB_INITIO}{article}{
      author={Doria, A.~V.},
      author={Hassanzadeh, S.~H.},
      author={Simis, A.},
       title={A characteristic-free criterion of birationality},
        date={2012},
     journal={Adv. Math.},
      volume={230},
      number={1},
       pages={390\ndash 413},
}

\bib{EISEN_COMM}{book}{
      author={Eisenbud, David},
       title={Commutative algebra with a view towards algebraic geometry},
      series={Graduate Texts in Mathematics, 150},
   publisher={Springer-Verlag},
        date={1995},
}

\bib{EH_SPECIALIZ}{article}{
      author={Eisenbud, David},
      author={Huneke, Craig},
       title={Cohen-{M}acaulay {R}ees algebras and their specialization},
        date={1983},
        ISSN={0021-8693},
     journal={J. Algebra},
      volume={81},
      number={1},
       pages={202\ndash 224},
}

\bib{FLENNER_O_CARROLL_VOGEL}{book}{
      author={Flenner, H.},
      author={O'Carroll, L.},
      author={Vogel, W.},
       title={Joins and intersections},
      series={Springer Monographs in Mathematics},
   publisher={Springer-Verlag, Berlin},
        date={1999},
}

\bib{FLENNER_MANARESI}{article}{
      author={Flenner, Hubert},
      author={Manaresi, Mirella},
       title={A numerical characterization of reduction ideals},
        date={2001},
        ISSN={0025-5874},
     journal={Math. Z.},
      volume={238},
      number={1},
       pages={205\ndash 214},
         url={https://doi.org/10.1007/PL00004900},
      review={\MR{1860742}},
}

\bib{FULTON_INTERSECTION_THEORY}{book}{
      author={Fulton, William},
       title={Intersection theory},
     edition={Second},
      series={Ergebnisse der Mathematik und ihrer Grenzgebiete. 3. Folge. A
  Series of Modern Surveys in Mathematics [Results in Mathematics and Related
  Areas. 3rd Series. A Series of Modern Surveys in Mathematics]},
   publisher={Springer-Verlag, Berlin},
        date={1998},
      volume={2},
}

\bib{GAFFNEY}{incollection}{
      author={Gaffney, Terence},
       title={Generalized {B}uchsbaum-{R}im multiplcities and a theorem of
  {R}ees},
        date={2003},
      volume={31},
       pages={3811\ndash 3827},
         url={https://doi.org/10.1081/AGB-120022444},
        note={Special issue in honor of Steven L. Kleiman},
      review={\MR{2007386}},
}

\bib{GG}{article}{
      author={Gaffney, Terence},
      author={Gassler, Robert},
       title={Segre numbers and hypersurface singularities},
        date={1999},
        ISSN={1056-3911},
     journal={J. Algebraic Geom.},
      volume={8},
      number={4},
       pages={695\ndash 736},
      review={\MR{1703611}},
}

\bib{GORTZ_WEDHORN}{book}{
      author={G\"{o}rtz, Ulrich},
      author={Wedhorn, Torsten},
       title={Algebraic geometry {I}. {S}chemes---with examples and exercises},
     edition={Second},
      series={Springer Studium Mathematik---Master},
   publisher={Springer Spektrum, Wiesbaden},
        date={[2020] \copyright 2020},
        ISBN={978-3-658-30732-5; 978-3-658-30733-2},
         url={https://doi.org/10.1007/978-3-658-30733-2},
      review={\MR{4225278}},
}

\bib{MACAULAY2}{misc}{
      author={Grayson, Daniel~R.},
      author={Stillman, Michael~E.},
       title={Macaulay2, a software system for research in algebraic geometry},
        note={Available at \url{http://www.math.uiuc.edu/Macaulay2/}},
}

\bib{EGAII}{article}{
      author={Grothendieck, A.},
       title={\'{E}l\'{e}ments de g\'{e}om\'{e}trie alg\'{e}brique. {II}.
  \'{E}tude globale \'{e}l\'{e}mentaire de quelques classes de morphismes},
        date={1961},
        ISSN={0073-8301},
     journal={Inst. Hautes \'{E}tudes Sci. Publ. Math.},
      number={8},
       pages={222},
         url={http://www.numdam.org/item?id=PMIHES_1961__8__222_0},
      review={\MR{217084}},
}

\bib{HARRIS}{book}{
      author={Harris, Joe},
       title={Algebraic geometry},
      series={Graduate Texts in Mathematics},
   publisher={Springer-Verlag, New York},
        date={1995},
      volume={133},
        note={A first course, Corrected reprint of the 1992 original},
}

\bib{HARTSHORNE}{book}{
      author={Hartshorne, Robin},
       title={Algebraic geometry},
   publisher={Springer-Verlag, New York-Heidelberg},
        date={1977},
        note={Graduate Texts in Mathematics, No. 52},
}

\bib{HERMANN_MULTIGRAD}{article}{
      author={Herrmann, Manfred},
      author={Hyry, Eero},
      author={Ribbe, J\"{u}rgen},
      author={Tang, Zhongming},
       title={Reduction numbers and multiplicities of multigraded structures},
        date={1997},
     journal={J. Algebra},
      volume={197},
      number={2},
       pages={311\ndash 341},
}

\bib{HOCHSTER_ROBERTS_INVARIANTS}{article}{
      author={Hochster, Melvin},
      author={Roberts, Joel~L.},
       title={Rings of invariants of reductive groups acting on regular rings
  are {C}ohen-{M}acaulay},
        date={1974},
     journal={Advances in Math.},
      volume={13},
       pages={115\ndash 175},
}

\bib{Huh12}{article}{
      author={Huh, June},
       title={Milnor numbers of projective hypersurfaces and the chromatic
  polynomial of graphs},
        date={2012},
     journal={Journal of the American Mathematical Society},
      volume={25},
      number={3},
       pages={907\ndash 927},
}

\bib{huneke2006integral}{book}{
      author={Huneke, Craig},
      author={Swanson, Irena},
       title={Integral closure of ideals, rings, and modules},
   publisher={Cambridge University Press},
        date={2006},
      volume={13},
}

\bib{HYRY_MULTIGRAD}{article}{
      author={Hyry, Eero},
       title={The diagonal subring and the {C}ohen-{M}acaulay property of a
  multigraded ring},
        date={1999},
     journal={Trans. Amer. Math. Soc.},
      volume={351},
      number={6},
       pages={2213\ndash 2232},
}

\bib{VERMA_BIGRAD}{incollection}{
      author={Katz, D.},
      author={Mandal, S.},
      author={Verma, J.~K.},
       title={Hilbert functions of bigraded algebras},
        date={1994},
   booktitle={Commutative algebra ({T}rieste, 1992)},
   publisher={World Sci. Publ., River Edge, NJ},
       pages={291\ndash 302},
}

\bib{KSU}{incollection}{
      author={Kennedy, Gary},
      author={Simis, Aron},
      author={Ulrich, Bernd},
       title={Specialization of {R}ees algebras with a view to tangent star
  algebras},
        date={1994},
   booktitle={Commutative algebra ({T}rieste, 1992)},
   publisher={World Sci. Publ., River Edge, NJ},
       pages={130\ndash 139},
}

\bib{KLEIMAN_THORUP_GEOM}{article}{
      author={Kleiman, Steven},
      author={Thorup, Anders},
       title={A geometric theory of the {B}uchsbaum-{R}im multiplicity},
        date={1994},
        ISSN={0021-8693},
     journal={J. Algebra},
      volume={167},
      number={1},
       pages={168\ndash 231},
         url={https://doi-org.kuleuven.e-bronnen.be/10.1006/jabr.1994.1182},
      review={\MR{1282823}},
}

\bib{KLEIMAN_THORUP_MIXED}{article}{
      author={Kleiman, Steven},
      author={Thorup, Anders},
       title={Mixed {B}uchsbaum-{R}im multiplicities},
        date={1996},
        ISSN={0002-9327},
     journal={Amer. J. Math.},
      volume={118},
      number={3},
       pages={529\ndash 569},
  url={http://muse.jhu.edu.kuleuven.e-bronnen.be/journals/american_journal_of_mathematics/v118/118.3kleiman.pdf},
      review={\MR{1393259}},
}

\bib{KNUTSON_MILLER_SCHUBERT}{article}{
      author={Knutson, Allen},
      author={Miller, Ezra},
       title={Gr\"{o}bner geometry of {S}chubert polynomials},
        date={2005},
     journal={Ann. of Math. (2)},
      volume={161},
      number={3},
       pages={1245\ndash 1318},
}

\bib{KNUTSON_MILLER_YONG}{article}{
      author={Knutson, Allen},
      author={Miller, Ezra},
      author={Yong, Alexander},
       title={Gr\"{o}bner geometry of vertex decompositions and of flagged
  tableaux},
        date={2009},
        ISSN={0075-4102},
     journal={J. Reine Angew. Math.},
      volume={630},
       pages={1\ndash 31},
         url={https://doi.org/10.1515/CRELLE.2009.033},
      review={\MR{2526784}},
}

\bib{KPU_blowup_fibers}{article}{
      author={Kustin, Andrew},
      author={Polini, Claudia},
      author={Ulrich, Bernd},
       title={Blowups and fibers of morphisms},
        date={2016},
     journal={Nagoya Math. J.},
      volume={224},
      number={1},
       pages={168\ndash 201},
}

\bib{LIPMAN_EQUIMULT}{incollection}{
      author={Lipman, Joseph},
       title={Equimultiplicity, reduction, and blowing up},
        date={1982},
   booktitle={Commutative algebra ({F}airfax, {V}a., 1979)},
      series={Lect. Notes Pure Appl. Math.},
      volume={68},
   publisher={Dekker, New York},
       pages={111\ndash 147},
      review={\MR{655801}},
}

\bib{MATSUMURA}{book}{
      author={Matsumura, Hideyuki},
       title={Commutative ring theory},
     edition={1},
      series={Cambridge Studies in Advanced Mathematics volume 8},
   publisher={Cambridge University Press},
        date={1989},
}

\bib{MCADAM}{book}{
      author={McAdam, Stephen},
       title={Asymptotic prime divisors},
      series={Lecture Notes in Mathematics},
   publisher={Springer-Verlag, Berlin},
        date={1983},
      volume={1023},
        ISBN={3-540-12722-4},
         url={https://doi.org/10.1007/BFb0071575},
      review={\MR{722609}},
}

\bib{EXPONENTIAL_VARIETIES}{article}{
      author={Micha{\l}ek, Mateusz},
      author={Sturmfels, Bernd},
      author={Uhler, Caroline},
      author={Zwiernik, Piotr},
       title={Exponential varieties},
        date={2016},
     journal={Proc. Lond. Math. Soc. (3)},
      volume={112},
      number={1},
       pages={27\ndash 56},
}

\bib{MS}{book}{
      author={Miller, Ezra},
      author={Sturmfels, Bernd},
       title={Combinatorial commutative algebra},
      series={Graduate Texts in Mathematics},
   publisher={Springer-Verlag, New York},
        date={2005},
      volume={227},
        ISBN={0-387-22356-8},
      review={\MR{2110098}},
}

\bib{PTUV}{article}{
      author={Polini, Claudia},
      author={Trung, Ngo~Viet},
      author={Ulrich, Bernd},
      author={Validashti, Javid},
       title={Multiplicity sequence and integral dependence},
        date={2020},
        ISSN={0025-5831},
     journal={Math. Ann.},
      volume={378},
      number={3-4},
       pages={951\ndash 969},
         url={https://doi.org/10.1007/s00208-020-02059-5},
      review={\MR{4163518}},
}

\bib{RATLIFF_QUNMIXED_II}{article}{
      author={Ratliff, L.~J., Jr.},
       title={On quasi-unmixed local domains, the altitude formula, and the
  chain condition for prime ideals. {II}},
        date={1970},
        ISSN={0002-9327,1080-6377},
     journal={Amer. J. Math.},
      volume={92},
       pages={99\ndash 144},
         url={https://doi.org/10.2307/2373501},
      review={\MR{265339}},
}

\bib{REES}{article}{
      author={Rees, D.},
       title={{${\mathfrak{a}}$}-transforms of local rings and a theorem on
  multiplicities of ideals},
        date={1961},
        ISSN={0008-1981},
     journal={Proc. Cambridge Philos. Soc.},
      volume={57},
       pages={8\ndash 17},
         url={https://doi.org/10.1017/s0305004100034800},
      review={\MR{118750}},
}

\bib{SUV_MULT}{article}{
      author={Simis, Aron},
      author={Ulrich, Bernd},
      author={Vasconcelos, Wolmer~V.},
       title={Codimension, multiplicity and integral extensions},
        date={2001},
     journal={Math. Proc. Cambridge Philos. Soc.},
      volume={130},
      number={2},
       pages={237\ndash 257},
}

\bib{stacks-project}{misc}{
      author={{Stacks project authors}, The},
       title={The stacks project},
         how={\url{https://stacks.math.columbia.edu}},
        date={2024},
}

\bib{STUCKRAD_VOGEL_BUCHSBAUM_RINGS}{book}{
      author={St\"{u}ckrad, J\"{u}rgen},
      author={Vogel, Wolfgang},
       title={Buchsbaum rings and applications},
   publisher={Springer-Verlag, Berlin},
        date={1986},
        ISBN={3-540-16844-3},
         url={https://doi.org/10.1007/978-3-662-02500-0},
        note={An interaction between algebra, geometry and topology},
      review={\MR{881220}},
}

\bib{STURMFELS_UHLER}{article}{
      author={Sturmfels, Bernd},
      author={Uhler, Caroline},
       title={Multivariate {G}aussian, semidefinite matrix completion, and
  convex algebraic geometry},
        date={2010},
        ISSN={0020-3157},
     journal={Ann. Inst. Statist. Math.},
      volume={62},
      number={4},
       pages={603\ndash 638},
         url={https://doi.org/10.1007/s10463-010-0295-4},
      review={\MR{2652308}},
}

\bib{TEISSIER_RES2}{article}{
      author={Teissier, B.},
       title={Résolution simultanée : {II} - résolution simultanée et
  cycles évanescents},
    language={fre},
        date={1976-1977},
     journal={Séminaire sur les singularités des surfaces},
       pages={1\ndash 66},
         url={http://eudml.org/doc/114142},
}

\bib{TEISSIER_CYC}{incollection}{
      author={Teissier, Bernard},
       title={Cycles \'{e}vanescents, sections planes et conditions de
  {W}hitney},
        date={1973},
   booktitle={Singularit\'{e}s \`a {C}arg\`ese ({R}encontre {S}ingularit\'{e}s
  {G}\'{e}om. {A}nal., {I}nst. \'{E}tudes {S}ci., {C}arg\`ese, 1972)},
      series={Ast\'{e}risque},
      volume={Nos. 7 et 8},
   publisher={Soc. Math. France, Paris},
       pages={285\ndash 362},
      review={\MR{374482}},
}

\bib{THORUP}{incollection}{
      author={Thorup, Anders},
       title={Rational equivalence theory on arbitrary {N}oetherian schemes},
        date={1990},
   booktitle={Enumerative geometry ({S}itges, 1987)},
      series={Lecture Notes in Math.},
      volume={1436},
   publisher={Springer, Berlin},
       pages={256\ndash 297},
         url={https://doi.org/10.1007/BFb0084049},
      review={\MR{1068969}},
}

\bib{Trung2001}{article}{
      author={Trung, Ng\^{o}~Vi\^{e}t},
       title={Positivity of mixed multiplicities},
        date={2001},
     journal={Math. Ann.},
      volume={319},
      number={1},
       pages={33\ndash 63},
}

\bib{TRUNG_VERMA}{article}{
      author={Trung, Ngo~Viet},
      author={Verma, Jugal},
       title={Mixed multiplicities of ideals versus mixed volumes of
  polytopes},
        date={2007},
     journal={Trans. Amer. Math. Soc.},
      volume={359},
      number={10},
       pages={4711\ndash 4727},
}

\bib{UV_CRIT_MOD}{article}{
      author={Ulrich, Bernd},
      author={Validashti, Javid},
       title={A criterion for integral dependence of modules},
        date={2008},
        ISSN={1073-2780},
     journal={Math. Res. Lett.},
      volume={15},
      number={1},
       pages={149\ndash 162},
         url={https://doi.org/10.4310/MRL.2008.v15.n1.a13},
      review={\MR{2367181}},
}

\bib{UV_NUM_CRIT}{article}{
      author={Ulrich, Bernd},
      author={Validashti, Javid},
       title={Numerical criteria for integral dependence},
        date={2011},
        ISSN={0305-0041},
     journal={Math. Proc. Cambridge Philos. Soc.},
      volume={151},
      number={1},
       pages={95\ndash 102},
         url={https://doi.org/10.1017/S0305004111000144},
      review={\MR{2801316}},
}

\bib{VAN_DER_WAERDEN}{inproceedings}{
      author={Van~der Waerden, Bartel~Leendert},
       title={On {H}ilbert’s function, series of composition of ideals and a
  generalization of the theorem of {B}ezout},
        date={1929},
   booktitle={Proc. roy. acad. amsterdam},
      volume={31},
       pages={749\ndash 770},
}

\bib{vasconcelos1994arithmetic}{book}{
      author={Vasconcelos, Wolmer~V},
       title={Arithmetic of blowup algebras},
   publisher={Cambridge University Press},
        date={1994},
      volume={195},
}

\end{biblist}
\end{bibdiv}

\end{document}